\documentclass[11pt, reqno]{amsart}

\usepackage{amsmath,amsfonts,amssymb,mathrsfs}
\usepackage{mathrsfs,extpfeil}
\usepackage{indentfirst, latexsym, enumerate}

\usepackage{graphicx,epsfig,subcaption,float}

\usepackage{colortbl,bm}
\newcommand{\bbr}{\mathbb R}

\newcommand{\bbz}{\mathbb Z}

\usepackage{hyperref}

\usepackage{cite}
\usepackage{upgreek}
\hypersetup{
	hyperindex=true
}
\hypersetup{hidelinks}
\title[Nonlocal Kuramoto model with non-integrable kernels]{Relaxation dynamics of the continuum Kuramoto model with non-integrable kernels}

\author[L. Chen]{Li Chen}
\address[Li Chen]{\newline School of Business Informatics and Mathematics
\newline University of Mannheim, Mannheim 68159, Germany}
\email{li.chen@uni-mannheim.de}

\author[S.-Y. Ha]{Seung-Yeal Ha}
\address[Seung-Yeal Ha]{\newline Department of Mathematical Sciences and Research Institute of Mathematics \newline  Seoul National University, Seoul 08826, Republic of Korea}
\email{syha@snu.ac.kr}

\author[X. Wang]{Xinyu Wang$^*$}
\address[Xinyu Wang]{\newline Department of Mathematical Sciences \newline  Seoul National University, Seoul 08826, Republic of Korea\newline School of Mathematics \newline Harbin Institute of Technology, Harbin  150001, People's Republic of China}
\email{wangxinyu97@snu.ac.kr}

\author[V. Zhidkova]{Valeriia Zhidkova}
\address[Valeriia Zhidkova]{\newline School of Business Informatics and Mathematics
\newline University of Mannheim, Mannheim 68159, Germany}
\email{valeriia.zhidkova@uni-mannheim.de}

\newtheorem{theorem}{Theorem}[section]
\newtheorem{lemma}{Lemma}[section]

\newtheorem{proposition}{Proposition}[section]

\newtheorem{remark}{Remark}[section]

\newtheorem{definition}{Definition}[section]

\def\charf {\mbox{{\text 1}\kern-.24em {\text l}}}

\newcommand*\di{\mathop{}\!\mathrm{d}}
\allowdisplaybreaks[4]

\begin{document}



\thanks{\textbf{Acknowledgment.}
The work of S.-Y. Ha is supported by National Research Foundation(NRF) grant funded by the Korea government(MIST) (RS-2025-00514472), and the work of X. Wang is  partially supported by the Natural Science Foundation of China (grants 123B2003), the China Postdoctoral Science Foundation (grants 2025M774290), and Heilongjiang Province Postdoctoral Funding (grants  LBH-Z24167). Li Chen’s work is partially supported by Deutsche Forschungsgemeinschaft (DFG, German Research Foundation – 547277619) and the National Natural Science Foundation of China (12171218). The work of V. Zhidkova is partially supported by Deutsche Forschungsgemeinschaft (DFG, German Research Foundation – 547277619).}
\thanks{* Corresponding author.}

\begin{abstract}
We study the asymptotic behavior of the continuum Kuramoto model with a fractional Laplacian-type  kernel. For this, we construct global weak solutions via a two-parameter regularization procedure using a kernel truncation with fractional dissipation. Using a priori uniform estimates derived in fractional Sobolev spaces, we employ compactness arguments to construct  global weak solutions to the singular continuum Kuramoto model. Furthermore, we also establish an exponential relaxation toward the initial phase average in $L^2$-norm under suitable assumptions on initial data and system parameters. These findings provide a rigorous characterization of the existence of solutions and the emergent dynamics of Kuramoto ensembles under physically important strongly singular interactions, including power-law singular kernels and Coulomb-type  kernels.
\end{abstract}

\maketitle

%
%
%
%
%
%
\section{Introduction}\label{sec:1}
\setcounter{equation}{0}
Synchronization denotes phenomena in which weakly coupled oscillators adjust their rhythms, and it is one of the collective behaviors appearing in complex systems. To model such phenomena, Arthur Winfree and Yoshiki Kuramoto introduced phase-coupled models for limit-cycle oscillators in \cite{Wi, Kuramoto1975}. The phase model proposed by Kuramoto in 1975, is a canonical mathematical framework for studying synchronization phenomena in systems of coupled oscillators \cite{Kuramoto1975}. It describes the evolution of phase oscillators driven by intrinsic frequencies and mutual coupling interactions, and has become a fundamental model in nonlinear dynamics, statistical physics, and applied mathematics. Since its introduction, the Kuramoto model and its variants have been extensively studied in both finite-dimensional and continuum settings; see, for instance, \cite{A-B2005, A-B2019, B-P2012, D-B2014, HaKo2016, T-T1998, B-H2010} for comprehensive surveys. In this paper, we focus on the continuum Kuramoto model \cite{C-H-W-Z-2025, KoHa2025, H-K-P-R-S, Tr-1, Tr-2} which can be obtained from the lattice Kuramoto model in a continuum limit, and study the global well-posedness and long-time behaviors of the continuum Kuramoto model with nonlocal non-integrable singularity.

To set up the stage, we begin with a brief introduction of the lattice Kuramoto model. Let $\Omega$ be a bounded Lipschitz domain in $\bbr^d$ with a positive measure, and $\Gamma\subset\Omega$ be a uniform regular lattice with the index set $\Lambda \subset \bbz^d$ and the same lattice spacing in each coordinate. Since $\Omega$ is compact and $\Gamma$ is discrete, the index set $\Lambda$ is a finite set.  Let $\theta_\alpha = \theta_\alpha(t)$ and $\nu_\alpha$ be the phase and natural frequency of the internal (Kuramoto) oscillator at the lattice point $x_\alpha \in \Gamma$ and at time $t$. We assume that the communication weight between oscillators located at $x_{\alpha}$ and $x_{\beta}$ is given by the nonnegative real value $\psi_{\alpha \beta}$:
\[ \psi_{\alpha \beta} = \psi(|x_\alpha - x_\beta|), \quad \alpha, \beta \in \Lambda.\]
Here, $\psi: = \psi(\cdot, \cdot)~$ is a nonnegative kernel function. In this setting, the dynamics of $\theta_{\alpha}$ is governed by the Cauchy problem to the lattice Kuramoto model with network topology $(\psi_{\alpha \beta})$:
\begin{equation} \label{A-1}
	\begin{cases}
		\displaystyle \dot{\theta}_\alpha = \nu_\alpha + \frac{\kappa}{|\Lambda|} \sum_{\beta\in\Lambda} \psi_{\alpha \beta} \sin\big(\theta_{\beta} - \theta_\alpha \big), \quad t>0,\\
		\displaystyle \theta_\alpha \Big|_{t = 0}  = \theta_\alpha^{\mathrm{in}}, \quad\alpha \in\Lambda, 
		\end{cases}
\end{equation}
where $\kappa$ is the nonnegative coupling strength. The emergent dynamics of the Cauchy problem \eqref{A-1} has been extensively studied in literature; e.g., \cite{C-S2023, D-B2012, D-B2011, HaKim2016,D-X,H-L-X} and survey articles \cite{Erm-1985, H-R, D-B2014, A-B2005}. 

Note that as the number of oscillators tends to infinity, it is a fundamental problem in statistical physics and applied analysis to derive and analyze the corresponding limiting dynamics. As far as the authors know, there are three main approaches to describing the emergent behavior of infinite particle systems. The first approach concerns all-to-all interaction networks, where each agent interacts with every other agent. In this mean-field coupling regime, the limiting dynamics is described by the Vlasov-McKean equation in the mean-field limit \cite{C-C-H-K-K,Poyato2019, B-C-M, w6, w9, Lucon2014}. The second approach deals with sparse or non-complete coupling structures in densely connected networks, where the dynamics can be approximated by integro-differential equations through graph limit \cite{ChoHa2023, B-D2022, KoHa2025, C-H-W-Z-2025}. The third approach studies infinite systems of ordinary differential equations posed directly on infinite graphs, without passing to kinetic or continuum limits \cite{Bra-2019, w1, w5,w11,J1}. In this work, we adopt the second approximation methodology and focus on the continuum Kuramoto model with a nonlocal non-integrable kernel, given by
\begin{equation}\label{A-2}
	\begin{cases}
		\partial_t \theta(t,x)
		= \nu(x) + \kappa\, \mathrm{P.V.}\displaystyle\int_{\Omega} 
		\psi(x,y)\,\sin\big(\theta(t,y)-\theta(t,x)\big) \di y, 
		\quad t>0,\ x\in\Omega,\\[3ex]
		\theta \Big|_{t = 0}=\theta^{\mathrm{in}}.
	\end{cases}
\end{equation}
Here $\psi:\Omega\times\Omega \setminus \lbrace (x,x)\in \Omega \times \Omega \rbrace \to[0,\infty)$ is a non-integrable interaction kernel:
\begin{equation*}
	\psi(x,y):=\frac{1}{|x-y|^{d+2s}}, \qquad 0<s<1.
\end{equation*}
The principal value integral is understood as follows.
\begin{align*}
\begin{aligned}
& \mathrm{P.V.}\int_{\Omega} \psi(x,y)\sin\big(\theta(t,y)-\theta(t,x)\big) \di y \\
& \hspace{2cm} :=\lim_{\rho\to0}\int_{\Omega\cap\{|y-x|>\rho\}}
	\psi(x,y)\sin\big(\theta(t,y)-\theta(t,x)\big) \di y,
\end{aligned}
\end{align*}
as long as the right-hand side exists. If there is no confusion, we may drop P.V. in what follows. Throughout this paper, we consider a bounded Lipschitz domain $\Omega \subset \mathbb{R}^d, d\ge 1$.

While Kuramoto models with singular phase interactions have been extensively studied, see for instance \cite{P-P-J2021,Poyato2019,C-H-W-Z-2025}, where finite-time synchronization was established, non-integrable kernels have so far only been investigated in the linear framework, notably in the theory of the fractional Laplacian \cite{C-R-Sire2010,C-L-L2017,C-K-S-2010,R-S-2014}. To the best of our knowledge, the nonlinear Kuramoto model with nonlocal non-integrable spatial interactions has not been studied in literature.  Thus, we address the following two questions for \eqref{A-2}: 
\vspace{0.2cm}
\begin{itemize}
	\item[(Q1)]:~Under what conditions can we establish the existence of weak solutions to \eqref{A-2}?
	\vspace{0.1cm}
	\item[(Q2)]:~If such weak solutions exist, can we show the emergence of collective behaviors for \eqref{A-2}?
\end{itemize}
\vspace{0.2cm}

The main purpose of this paper is to answer the aforementioned questions. For (Q1), we recall the concept of weak solution to \eqref{A-2}. 
\begin{definition}\label{D1.1}
\begin{enumerate}
\item
	We say that $\theta$ is a weak solution to \eqref{A-2} if for every test function
	$\varphi\in C_c^\infty([0,T)\times\Omega)$, it holds
	\begin{align*}
		&\int_0^T\!\!\int_\Omega \theta(t,x)\,\partial_t\varphi(t,x) \di x\di t
		+ \int_\Omega \theta^{\rm in}(x)\varphi(0,x) \di x \notag\\
		&\hspace{0.5cm} + \int_0^T \int_\Omega \nu(x) \varphi (t,x) \di x \di t + \kappa\int_0^T\!\!\iint_{\Omega\times \Omega}
		\frac{\sin(\theta(t,y)-\theta(t,x))}{|x-y|^{d+2s}}
		\,\varphi(t,x) \di y\di x\di t = 0.
	\end{align*}
\vspace{0.1cm}	
\item
(Essential diameter): Let $\Omega\subset\mathbb R^d$ be a measurable set and let $f:\Omega\to\mathbb R$ be a measurable function. 
	The \emph{essential diameter} of $f$ on $\Omega$ is defined by
	\begin{equation*}
		{\mathcal D}[f] := \operatorname*{ess\,sup}_{x\in \Omega} f- \operatorname*{ess\,inf}_{x\in \Omega}f.
	\end{equation*}
\end{enumerate}
\end{definition}
\noindent Note that the second term on the right-hand side of \eqref{A-2} is no longer bounded due to the nonlocal singular effect. Therefore, we cannot directly obtain the existence of solutions by the standard Cauchy--Lipschitz theory. Moreover, we cannot control the phase diameter functional using the existing  methods in \cite{C-H-W-Z-2025,HaKim2016,KoHa2025}. To overcome these difficulties, we consider a suitably designed double regularized equation. Specifically, we replace the kernel $\psi$ by its truncated version $\psi_\varepsilon (x,y) = (|x-y|+\varepsilon)^{-d-2s}$ and incorporate a vanishing dissipative term associated with the regional fractional Laplacian. This construction yields a well-posed initial value problem \eqref{B-3-0}. For the solution of this Cauchy problem \eqref{B-3-0}, we are able to control the phase diameter and, in fact, obtain uniform bounds in $L^\infty$-norm. Importantly, these structural properties are preserved in the limit, as the regularization parameters vanish. This observation is crucial as it allows us to apply the energy method in combination with the Aubin-Lions compactness lemma to establish the existence of solutions. More precisely, we assume that initial data and frequency satisfy
\begin{align} \label{A-3}
\nu(x)\equiv\nu,~\forall ~x\in \Omega, \quad  \theta^{\mathrm{in}}\in L^\infty(\Omega)\cap H^s(\Omega), \quad s\in (0,1),
\end{align}
and
\begin{equation*}
	{\mathcal D}[\theta^{\mathrm{in}}]
	:=\operatorname*{ess\,sup}_{x\in \Omega} \theta^{\mathrm{in}} (x)
	-\operatorname*{ess\,inf}_{x\in \Omega} \theta^{\mathrm{in}}(x)
	<\pi.
\end{equation*}
Then, there exists at least one weak solution $\theta\in L^2([0,\infty);H^s(\Omega))\cap C([0,\infty);L^2(\Omega))$ to \eqref{A-2} such that 
\begin{equation}\label{A-4}
	{\mathcal D}[\theta(t)]\le {\mathcal D}[\theta^{\mathrm{in}}] <\pi,
	\qquad  t\ge0.
\end{equation}
For more details, we refer to Theorem~\ref{T2.1}.\newline

For (Q2), we study solutions satisfying the phase condition \eqref{A-4} obtained in Theorem~\ref{T2.1}. For this, we combine the monotonicity property of the function $\sin x/x$ on $[0,\pi)$ with the fractional Poincaré inequality to obtain exponential convergence to equilibrium in $L^2$-norm. Again, the control of the phase diameter is necessary since it provides the monotonicity of $\sin x /x$. More precisely, suppose that $\theta$ is a weak solution to \eqref{A-2} with initial data
\[
\theta^\mathrm{in} \in L^\infty(\Omega)\cap H^s(\Omega)
\quad \text{and} \quad 
{\mathcal D}[\theta^\mathrm{in}] = M <\pi.
\]
Then there exists a constant $C_P=C_P(\Omega,d,s)>0$ such that, for all $t\ge0$,
\begin{equation*}
	\|\theta(t,\cdot)-\bar{\theta}\|_{L^2(\Omega)}^2
	\le \exp\!\big(-\kappa\,c_M\,C_P\,t\big)\;
	\|\theta^{\rm in}-\bar{\theta}\|_{L^2(\Omega)}^2,
	\qquad
	c_M:=\frac{\sin M}{M}>0,
\end{equation*}
where $\bar{\theta}$ denotes the initial average phase, defined by
\begin{equation}\label{A-5}
	\bar{\theta}:=\frac1{|\Omega|}\int_\Omega \theta^{\rm in}(x) \di x.
\end{equation}
For more details, we refer to Theorem~\ref{T2.2}.\newline

The rest of this paper is organized as follows. In Section \ref{sec:2}, we study preparatory materials such as conservation of average phase, regional fractional Laplacian to be used in the analysis of \eqref{A-2}. After we review the related previous results for \eqref{A-2} with regular and moderately singular kernels, we present a brief description of our main results. In Section \ref{sec:3}, we  establish a global existence of a weak solution via a double regularization technique. In Section \ref{sec:4}, we show that an exponential relaxation of the phase function toward the initial phase average emerges. Finally, Section \ref{sec:5} is devoted to a brief summary of the main results and discussion of some remaining issues for a future work. In Appendix \ref{App-A}, we study the global well-posedness of the doubly regularized equation. In Appendix \ref{App-B}, we provide a detailed proof of Lemma \ref{L3.5}.

\vspace{0.5cm}

\noindent {\bf Gallery of Notation.}~Let $\Omega \subset \mathbb{R}^d$ be a bounded Lipschitz domain and let $s \in (0,1)$. The fractional Sobolev space $W^{s,2}(\Omega)$ or $H^s(\Omega)$, also known as the \textit{Sobolev–Slobodeckij space}, is defined as
\[
H^{s}(\Omega) := \left\{ u \in L^2(\Omega) \;\middle|\; \iint_{\Omega \times \Omega} \frac{|u(x) - u(y)|^2}{|x - y|^{d + 2s}} \di x \di y < \infty \right\}.
\]
The associated norm on $H^{s}(\Omega)$ is given by
\[
\|u\|_{H^{s}(\Omega)}^2 := \|u\|_{L^2(\Omega)}^2 + \|u\|_{\dot H^{s}(\Omega)}^2,
\]
where the Gagliardo seminorm  is defined by
\begin{equation} \label{A-6}
\|u\|_{\dot H^{s}(\Omega)}^2 := \iint_{\Omega \times \Omega} \frac{|u(x) - u(y)|^2}{|x - y|^{d + 2s}} \di x \di y.
\end{equation}
We refer the reader to \cite{DiNezza2012} for more information on Sobolev-Slobodeckij spaces. Throughout this paper, for a set $A$, we write $|A|$ to denote the Lebesgue measure of the set $A$. \newline

Throughout the paper, we consider a bounded Lipschitz domain $\Omega \subset \bbr^d$. This domain is an extension domain for $H^s$ (or any $W^{s,p}$ with $p\in [1,+\infty)$ and $s\in (0,1)$, in fact), i.e., we can extend any function $u\in H^s(\Omega)$ to $H^s(\bbr^d)$. This fact is crucial for the proof of the compact embedding of $H^s(\Omega)$ into $L^2(\Omega)$ and the continuous embedding of $L^2(\Omega)$ into the dual $(H^s(\Omega))^*$. We refer to Theorem 5.4 in \cite{DiNezza2012} for more details.

\vspace{0.5cm}

\section{Preliminaries}\label{sec:2}
\setcounter{equation}{0}
In this section, we first study the basic properties of \eqref{A-2}, and then, we recall the basic properties of the fractional Laplacian operator and Aubin--Lions lemma.  Moreover, we review previous results of \eqref{A-2} under regular and weakly singular kernels. Finally, we summarize our main results.
\subsection{Conservation of average phase}\label{sec:2.1}
Suppose that the conditions \eqref{A-3} hold. Then, we integrate \eqref{A-2} over $x\in\Omega$ using Fubini's theorem to obtain
\[
\frac{\di}{\di t}\int_\Omega \theta(t,x)\di x
= \nu |\Omega| + \kappa\iint_{\Omega\times \Omega} 
\frac{\sin(\theta(t,y)-\theta(t,x))}{|x-y|^{d+2s}}\di y \di x.
\]
Since the sine function is odd and the kernel $|x-y|^{-d-2s}$ is symmetric, the integrand is antisymmetric under the exchange $(x,y)\mapsto (y,x)$. 
Hence, the double integral vanishes:
\[
\iint_{\Omega\times \Omega }
\frac{\sin(\theta(t,y)-\theta(t,x))}{|x-y|^{d+2s}} \di y \di x
=0.
\]
Therefore, for all $t\ge0$, one has 
\[ \frac{\di}{\di t} \left(
\frac{1}{|\Omega|}\int_\Omega \theta(t,x) \di x  \right)
= \nu.
\]

Now, we introduce the new variable 
\[
\tilde{\theta}(t,x):=\theta(t,x)-\bar{\theta}-\nu t
\]
to find 
\[
\frac{\di}{\di t} \left( \frac{1}{|\Omega|}\int_\Omega \tilde{\theta}(t,x) \di x \right) = 0.
\]
Thus, we have
\begin{align}\label{B-1}
	\frac{1}{|\Omega|}\int_\Omega \tilde{\theta}(t,x) \di x
	= \frac{1}{|\Omega|}\int_\Omega \tilde{\theta}(0,x) \di x = 0.
\end{align}
Hence, without loss of generality, we assume that
\begin{align}\label{B-2}
	\nu \equiv 0
	\quad \text{and} \quad 
	\frac{1}{|\Omega|} \int_\Omega \theta(t,x) \di x =\bar{\theta}=0,
\end{align}
where $\bar{\theta}$ is defined in \eqref{A-5}.
\subsection{The regional fractional Laplacian}\label{sec:2.2}
In this part, we recall the basic properties of the regional fractional Laplacian. \newline

Let $\Omega\subset\mathbb R^d$ be a bounded domain with Lipschitz boundary and $s\in(0,1)$.  
	We define the symmetric bilinear form $\mathcal{A}$ on $H^s(\Omega) \subset L^2(\Omega)$ as follows
	\begin{align}\label{B-3}
	\mathcal {A}(u,v)
	:=\frac12\iint_{\Omega\times\Omega}
	\frac{(u(x)-u(y))(v(x)-v(y))}{|x-y|^{d+2s}} \di x \di y,
	\qquad u,v\in H^s(\Omega).
	\end{align}

The space $H^s(\Omega)$ becomes a Hilbert space \cite{DiNezza2012} with the norm
\begin{align*}
\|u\|_{H^s(\Omega)}^2 = \|u\|_{L^2(\Omega)} + \mathcal{A}(u,u),
\end{align*}
so the bilinear form $\mathcal{A}$ is closed. By Kato's theorem, there exists a unique nonnegative definite self-adjoint operator $A$ on $L^2(\Omega)$ defined on its domain $D(A) \subset H^s(\Omega)$. It holds
\begin{align*}
\langle Au,v \rangle_{L^2(\Omega)} = \mathcal{A}(u,v), \quad u\in  D(A), \, v\in H^s(\Omega).
\end{align*}
For any $1$-Lipschitz function $\eta: \bbr \rightarrow \bbr$ and any $u \in H^s(\Omega)$ it obviously holds $\eta(u) \in H^s(\Omega)$ and
\begin{align*}
\mathcal{A}(\eta(u), \eta(u)) \leq \mathcal{A}(u,u).
\end{align*}
Therefore, $\mathcal{A}$ is a Dirichlet form. Then, the nonpositive operator $-A$ generates a Markovian symmetric contraction semigroup $(T(t))_t = (e^{-tA})_t$ acting on $L^2(\Omega)$. In other words, it holds
	\begin{align}\label{NewL11}
	\| T(t) \|_{\mathcal{L} (L^2, L^2)} \leq 1 \quad \mathrm{for} \ \mathrm{all} \ t\geq 0,
	\end{align}
	and
		\begin{align}\label{NewL12}
		\| T(t) \|_{\mathcal{L} (L^\infty, L^\infty)} \leq 1 \quad \mathrm{for} \ \mathrm{all} \ t\geq 0. 
		\end{align}
	We refer to \cite{Fukushima2010} for more details on semigroups generated by Dirichlet forms. Moreover, for more knowledge related to regional fractional Laplacian, we refer to \cite{C-R-Sire2010,C-L-L2017,C-K-S-2010,R-S-2014}.\newline

Next, we recall the Aubin--Lions compactness lemma and fractional Poincaré inequality, respectively.
\begin{lemma}[Aubin--Lions compactness lemma \cite{Brezis2010}]\label{L2.1}
	Let $X \subset\subset B \subset Y$ be Banach spaces with compact embedding from $X$ into $B$ and continuous embedding from $B$ into $Y$. Let $\{u_n\}$ be a sequence such that
	\[
	u_n \in L^p([0,T); X) \quad\text{and}\quad  \partial_t u_n \in L^q([0,T); Y),
	\]
	and these are uniformly bounded in $n$. If $1\le p<\infty$, then $\{u_n\}$ is relatively compact in $L^p([0,T); B)$. If $p=\infty$ and $q>1$, then $\{u_n\}$ is relatively compact in $C([0,T); B)$.
\end{lemma}
\begin{lemma}[Fractional Poincaré inequality \cite{H-V-2013}]\label{L2.2}
	Let $\Omega\subset\mathbb R^d, \, d\ge 1$ be a bounded domain and $s\in(0,1)$.  
	Then there exists a constant $C_P=C_P(\Omega,d,s)>0$ such that for all 
	$u\in H^s(\Omega)$ with zero mean, i.e., $\int_\Omega u(x) \di x=0$, it holds
	\[
	\|u\|_{L^2(\Omega)}^2 
	\le C_P \iint_{\Omega\times\Omega}
	\frac{|u(x)-u(y)|^2}{|x-y|^{d+2s}}~\di x\di y.
	\]
\end{lemma}
\begin{proof}
	Since in reference \cite{H-V-2013} they required dimension $d\ge2$, we provide a proof for any $d\ge1$ to fix setting in this paper. For every $x\in\Omega$, it holds
	\begin{align*}
		u(x) = \frac{1}{|\Omega|} \int_\Omega (u(x) - u(y)) \di y,
	\end{align*}
	since we have assumed a zero mean of $u$. Taking the square and using H\"older's inequality, we obtain
	\begin{align*}
		|u(x)|^2 = \left| \frac{1}{|\Omega|} \int_\Omega (u(x) - u(y)) \di y \right|^2 \leq \frac{1}{|\Omega|} \int_\Omega |u(x) - u(y)|^2 \di y.
	\end{align*}
	Integrating over $\Omega$, we get
	\begin{align} \label{ineq1}
		\int_\Omega |u(x)|^2 \di x\leq \frac{1}{|\Omega|} \iint_{\Omega \times \Omega} |u(x) - u(y)|^2 \di y \di x.
	\end{align}
	Now, using boundedness of $\Omega$, there exists a constant $C\ge 1$ such that
	\begin{align*}
		|x-y| \leq C \quad \forall ~x,y\in \Omega.
	\end{align*}
	For example, a possible choice is $C(\Omega) = \max\lbrace \mathrm{diam} \, \Omega, 1\rbrace$. With that, it follows
	\begin{align*}
		C^{d+2s}|x-y|^{-d-2s} \ge 1, \quad \forall ~x,y \in \Omega.
	\end{align*}
	And, by multiplying this inequality with a nonnegative term $|u(x)-u(y)|^2$, we easily obtain
	\begin{align*}
		|u(x)-u(y)|^2 \leq C^{d+2s}\frac{|u(x)-u(y)|^2 }{|x-y|^{d+2s}} \quad \mathrm{for} ~ \mathrm{a.e.} ~x,y\in \Omega. 
	\end{align*}
	Using this together with \eqref{ineq1} yields the claim:
	\begin{align*}
		\int_\Omega |u(x)|^2 \di x\leq \frac{1}{|\Omega|} \iint_{\Omega \times \Omega} |u(x) - u(y)|^2 \di y \di x \leq \frac{C^{d+2s}}{|\Omega|} \iint_{\Omega \times \Omega} \frac{|u(x)-u(y)|^2 }{|x-y|^{d+2s}} \di y \di x.
	\end{align*}
\end{proof}
\subsection{Previous results} \label{sec:2.3}
In this subsection, we briefly summarize the previous results on the emergent dynamics of the continuum Kuramoto model  \eqref{A-2} with regular and weakly singular interaction kernels. For $\Omega=[0,1],$  if the interaction kernel $\psi$ is bounded and measurable on $[0,1] \times [0, 1]$, then the continuum Kuramoto model \eqref{A-2} can be rigorously derived as the graph limit of the Kuramoto model on dense graphs (see \cite{Med,Medvedev2014}). \newline

Next, we recall two emergent dynamics of \eqref{A-2} on the bounded domain $\Omega = [0, 1]$ and unbounded domain $\Omega = {\mathbb R}^d$, respectively, without proofs. Suppose that the smooth natural frequency function $\nu \in C^1([0,1])$ satisfies the zero sum condition:
\begin{equation} \label{B-3-111}
\| \nu \|_{L^\infty(\Omega)} := \max_{x \in [0,1]} |\nu(x)| > 0 \quad \mbox{and} \quad  \int_0^1 \nu(x)\,\mathrm{d}x = 0.
\end{equation}
For such $\nu$ satisfying \eqref{B-3-111}, we define
\begin{equation}\label{41}
\begin{cases}
\displaystyle \Delta(x) := \frac{\nu(x)}{ \| \nu \|_{L^\infty(\Omega)}} \in [-1,1], \quad 
	r_\ast := \int_0^1 \sqrt{1 - \Delta^2(x)}\,\mathrm{d}x > 0, \\
\displaystyle \mathcal{D}[\nu] := \max_{x,y \in [0,1]} |\nu(x) - \nu(y)|,
	\quad 
	\lambda := \max_{x \in [0,1]} \left|\frac{\partial \nu(x)}{\partial x}\right|.
\end{cases}
\end{equation}
For $ \alpha > \max \left\{ D[\nu],\, \frac{ \| \nu \|_{L^\infty(\Omega)}  }{r_\ast}  \right\}$, let $D_1^\ast < D_2^\ast$ be two roots to the following trigonometric equation:
\begin{equation}\label{34}
	\sin x = \frac{\mathcal{D}[\nu]}{\alpha}, \quad  x \in (0, \pi).
\end{equation}
Then, it is easy to see that 
\[ 0 < D_1^\ast < \frac{\pi}{2} < D_2^\ast < \pi.  \]
\begin{proposition} 
\emph{\cite{L-L-X}} Suppose that natural frequency satisfies \eqref{B-3-111} and \eqref{41}. Assume that domain,  kernel, and initial datum $\theta^{\mathrm{in}} \in C^1([0,1])$ satisfy
\[ \Omega=[0,1],   \quad \psi\equiv 1, \quad \mathcal{D}[\theta^{\mathrm{in}}] \le D_2^\ast,  \]
and let $\theta = \theta(t,x)$ be a global  solution to \eqref{A-2}. Then, there exists an equilibrium $\theta^\ast$ such that 
\[ \lim_{t \to \infty} \| \theta(t) -  \theta^{\ast} \|_{L^{\infty}(\Omega)} = 0. \]
\end{proposition}
\begin{remark}
The existence and uniqueness/multiplicity of phase locked solution for continuum Kuramoto model were studied in \cite{Erm-1985, Tr-1, Tr-2}.
\end{remark}

\begin{proposition}
\emph{\cite{KoHa2025}}
Suppose that domain,  initial datum and kernel satisfy 
	\begin{equation*}\label{4.15}
	\begin{cases}
	 \displaystyle	\Omega=\mathbb{R}^d,  \quad \mathcal{D} [\theta^{\mathrm{in}}] < \pi, \quad  \|\psi\|_{\infty,1} := \sup_{x \in \mathbb{R}^d} \int_{\mathbb{R}^d} |\psi(x,y)|~\di y< \infty, \\
	 \displaystyle \exists~\tilde{\psi} \in L^1(\mathbb{R}^d)~~\mbox{such that}~~0 < \tilde{\psi}(y) \le\frac{ \psi(x,y)}
		{\int_{\mathbb{R}^d} \psi(x,z)\,\di z}, \quad \forall~ x,y \in \mathbb{R}^d, \\
	 \displaystyle \|\psi\|_{-\infty,1} := \inf_{x \in \mathbb{R}^d} \int_{\mathbb{R}^d} |\psi(x,y)|~\di y > 0.
	\end{cases}
	\end{equation*}
	Then, $\mathcal{D}[\theta]$ decays to zero exponentially fast:
	\[
	\mathcal D[\theta(t)]\le D[\theta^{\mathrm{in}}] e^{-\tilde{\gamma} t}, \quad \forall~ t > 0,
	\]
	where $\tilde{\gamma}$ is a positive constant given by
	\[
	\tilde{\gamma} := \min \left\{ \frac{1}{2}, \frac{\|\psi\|_{-\infty,1}}{\pi} \right\}
	\frac{\|\tilde{\psi}\|_{L^1(\mathbb{R}^d)}}{2\|\tilde{\psi}\|_{L^1(\mathbb{R}^d)} + 4\pi}.
	\]

\end{proposition}

\subsection{Description of main results} \label{sec:2.4}
In this subsection, we summarize our existence and long-time results to \eqref{A-2}, respectively. 
\begin{theorem}\label{T2.1}
\emph{(A global existence of a weak solution)}
	Let $\Omega\subset\mathbb R^d, d\ge1$ be a bounded domain with Lipschitz boundary, $s\in(0,1)$ and $\kappa>0$.  Suppose that natural frequency and initial datum satisfy
	\[
	\nu(x)\equiv \nu,~~\forall~ x\in\Omega, \quad 
	\theta^{\mathrm{in}}\in L^\infty(\Omega)\cap H^s(\Omega), \quad 
	{\mathcal D}[\theta^{\mathrm{in}}]<\pi.
	\]
		Then, there exists at least one global weak solution $\theta\in L^2([0,\infty);H^s(\Omega))\cap C([0,\infty];L^2(\Omega))$
	to \eqref{A-2} such that 
	\begin{align*}
		\sup_{0 \leq t < \infty} {\mathcal D}[\theta(t)] \le {\mathcal D}[\theta^{\mathrm{in}}] <\pi.
	\end{align*}
\end{theorem}
\begin{proof} Since the proof is very lengthy, we leave its detailed proof in next section, and we instead sketch the main steps of the proof on the global well-posedness to Cauchy problem \eqref{A-2} with a non-integrable kernel as follows. \newline

\noindent$\bullet $ \textbf{Step A (Doubly regularized equation)}:~Consider the Cauchy problem for the following doubly regularized equation:
\begin{equation}
	\begin{cases} \label{B-3-0}
		\displaystyle \partial_t\theta^{\varepsilon,\delta}(t,x)
		+\delta\,\displaystyle\int_{\Omega}
		\frac{\theta^{\varepsilon,\delta}(t,x)-\theta^{\varepsilon,\delta}(t,y)}{|x-y|^{d+2s}}\di y \\
		\displaystyle \hspace{1cm} =\kappa\displaystyle\int_\Omega \psi_\varepsilon(x,y)\sin\big(\theta^{\varepsilon,\delta}(t,y)-\theta^{\varepsilon,\delta}(t,x)\big)\di y, \quad t > 0,~~x \in \Omega, 
		\\[2ex]
		\displaystyle \theta^{\varepsilon,\delta} \Big|_{t = 0} =\theta^{\mathrm{in}}, \\[2ex]
		\displaystyle \psi_\varepsilon (x,y) = (|x-y|+\varepsilon)^{-d-2s}.
	\end{cases}
\end{equation}
Then, we show that the solution $\theta^{\varepsilon,\delta}$ satisfies a contraction property:
\begin{equation*}
	{\mathcal D}[\theta^{\varepsilon,\delta}(t)]
	\le {\mathcal D}[\theta^{\mathrm{in}}]<\pi,
	\qquad t\ge0.
\end{equation*}
For more details, we refer to Lemma \ref{L3.1} and Proposition \ref{PA.1} in Appendix \ref{App-A}.
\vspace{0.2cm}

\noindent$\bullet $ \textbf{Step B (Energy identity for $\theta^{\varepsilon,\delta}$)}:~We define truncated potential and kinetic energy functionals:
\begin{align}
\begin{aligned} \label{B-3-1}
& \mathcal E_{P,\varepsilon}[\theta^{\varepsilon,\delta}]
:=\frac{\kappa}{2}\iint_{\Omega\times\Omega}
\psi_\varepsilon(x,y)\,\big(1-\cos(\theta^{\varepsilon,\delta}(x)-\theta^{\varepsilon,\delta}(y))\big)\di x \di y, \\
& \mathcal E_{K}[\theta^{\varepsilon,\delta}]
:=\frac{\delta}{4}\iint_{\Omega\times\Omega}
\psi(x,y)\,(\theta^{\varepsilon,\delta}(x)-\theta^{\varepsilon,\delta}(y))^2 \di x \di y, \\
& {\mathcal E}_\varepsilon[\theta^{\varepsilon,\delta}] :=  \mathcal E_{P,\varepsilon}[\theta^{\varepsilon,\delta}] +  \mathcal E_{K}[\theta^{\varepsilon,\delta}].
\end{aligned}
\end{align}
Then, the total energy functional satisfies the energy identity (Lemma  \ref{L3.2}):
\begin{equation*}
{\mathcal E}_\varepsilon[\theta^{\varepsilon,\delta}(t)] +\int_0^t \|\partial_\tau\theta^{\varepsilon,\delta}(\tau)\|_{L^2(\Omega)}^2 \di \tau
	= {\mathcal E}_\varepsilon[\theta^{\mathrm{in}}], \quad t\ge0.
\end{equation*}
\vspace{0.2cm}

\noindent$\bullet $ \textbf{Step C (Convergence of subsequence $\left\{\theta^{\varepsilon_j,\delta}\right\}$ in $\varepsilon \to 0$)}:~By Aubin--Lions compactness lemma and energy estimate in Step B, for any fixed $\delta>0$, we can choose some subsequence $\varepsilon_j \to 0$ such that $\theta^{\varepsilon_j,\delta}\to \theta^{\delta}$. Then the limit $\theta^{\delta}$ satisfies the following regularized equation:
\begin{equation*}
\begin{cases}
\displaystyle \partial_t\theta^\delta(t,x)
	+\delta\,\int_{\Omega}\frac{\theta^\delta(t,x)-\theta^\delta(t,y)}{|x-y|^{d+2s}} \di y
	=\kappa\int_\Omega \frac{\sin\big(\theta^\delta(t,y)-\theta^\delta(t,x)\big)}{|x-y|^{d+2s}} \di y,~~t > 0,~~x \in \Omega, \\[1em]
\displaystyle \theta^\delta(0)=\theta^{\mathrm{in}},
\end{cases}
\end{equation*}
and a contraction property:
\begin{equation}\label{B-4}
	{\mathcal D}[\theta^{\delta}(t)]
	\le {\mathcal D}[\theta^{\mathrm{in}}] <\pi,
	\qquad t\ge0.
\end{equation}
We refer to Lemma \ref{L3.3} and Proposition \ref{PA.2} in Appendix \ref{App-A} for details.
\vspace{0.2cm}

\noindent$\bullet $ \textbf{Step D (Energy dissipation for $\theta^\delta$)}:~We use potential energy  \[
\mathcal E_{P}[\theta^{\delta}]
:=\frac{\kappa}{2}\iint_{\Omega\times\Omega}
\psi(x,y)\,\big(1-\cos(\theta^\delta(x)-\theta^\delta(y))\big)\,\di x\,\di y
\]
and \eqref{B-4} to obtain uniform estimates for $\theta^{\delta}$. We refer to section \ref{sec:3.4} for details.
\vspace{0.2cm}

\noindent$\bullet $ \textbf{Step E (Convergence of subsequence $\left\{\theta^{\delta_j}\right\}$  in $\delta \to 0$):}~By Aubin--Lions compactness lemma we can choose some subsequence $\delta_j \to 0$ such that $\theta^{\delta_j}\to \theta$ and $\theta$ is the desired weak solution to \eqref{A-2} (see Lemma \ref{L3.4} and Lemma \ref{L3.5}).
\end{proof}
With the above preparations, we are ready to state our long-time behavior results for the solution of \eqref{A-2}.
\begin{theorem}\label{T2.2}
\emph{(Exponential relaxation)}
	Let $\Omega\subset\mathbb R^d, \, d\ge 1$ be a bounded Lipschitz domain and $s\in(0,1)$. Fix $\kappa>0$ and let
	$\theta$ be a global weak solution to \eqref{A-2} on $[0,\infty)$
	with initial data $\theta^{\rm in}\in L^\infty(\Omega)\cap H^s(\Omega)$ satisfying 
	\begin{equation*}
	{\mathcal D}[\theta^{\rm in}]=:M<\pi,\qquad \forall~t\ge0.
	\end{equation*}
	Then there exists a constant $C_P=C_P(\Omega,d,s)>0$ such that
	for all $t\ge0$,
	\begin{equation*}
		\|\theta(t)-\bar{\theta}\|_{L^2(\Omega)}^2
		\le \exp\!\Big(-\kappa\,c_M\,C_P\,t\Big)\;
		\|\theta^{\rm in}-\bar{\theta}\|_{L^2(\Omega)}^2,
		\qquad
		c_M:=\frac{\sin M}{M}>0,
	\end{equation*}
	where $\bar{\theta}$ is the initial average phase defined as in \eqref{A-5}.
\end{theorem}
\begin{proof}
We combine the monotonicity property of the function $\sin x/x$ on $[0,\pi)$ with the fractional Poincaré inequality to find the desired estimates. We refer to Section \ref{sec:4} for more details.
\end{proof}
\begin{remark}
	 In reference \cite{Med,Medvedev2014,L-L-X,KoHa2025,C-H-W-Z-2025}, they consider the following four typical settings: the graph limit regime, bounded domains with uniform coupling, bounded domains with weakly singular integrable kernels, and spatially extended systems with regular integrable kernels. In contrast, our work focuses on a strongly singular, non-integrable 
	interaction kernel of fractional Laplacian type. In this regime, 
	the standard assumptions such as boundedness or integrability of the kernel 
	are no longer valid, and the existing analytical frameworks cannot be applied directly. To overcome these difficulties, we introduce a two-parameter regularization 
	procedure combined with fractional dissipation and compactness arguments 
	in fractional Sobolev spaces. As a result, we establish the global existence 
	of weak solutions and exponential relaxation toward the mean phase. Therefore, our results significantly extend the classical theory by providing 
	a rigorous description of the well-posedness and emergent synchronization 
	behavior of Kuramoto oscillators under strongly singular kernels, like non-integrable power-law singular kernels and Coulomb-type kernels.
\end{remark}

\vspace{0.5cm}

\section{Proof of Theorem \ref{T2.1}}\label{sec:3}
\setcounter{equation}{0}
In this section, we present detailed discussions outlined in the proof of Theorem \ref{T2.1}. Recall the fractional Laplacian kernel which is non-integrable:
\[
\psi(x,y) := \frac{1}{|x-y|^{d+2s}}, \qquad s\in(0,1),
\]
and the corresponding nonlocal continuum Kuramoto model:
\begin{equation}\label{C-1}
	\begin{cases}
		\partial_t \theta(t,x)
		= \kappa \displaystyle\int_\Omega \frac{1}{|x-y|^{d+2s}}
		\sin(\theta(t,y)-\theta(t,x))\,\di y,
		& t>0,\ x\in\Omega, \\[3ex]
		\theta \Big|_{t = 0} =\theta^{\mathrm{in}}.
	\end{cases}
\end{equation}
In what follows, we show the existence of a global weak solution to \eqref{C-1} via the steps delineated in Section \ref{sec:3.1}--Section \ref{sec:3.5}. 
\subsection{Doubly regularized equation}\label{sec:3.1}
In this subsection, we analyze the diameter of solution to the doubly regularized equation. For $\varepsilon>0$,  recall the truncated kernel
\[
\psi_\varepsilon(x,y):=\frac{1}{(|x-y|+\varepsilon)^{d+2s}},
\]
and the doubly regularized equation:
\begin{equation}\label{C-2}
	\begin{cases}
		\partial_t\theta^{\varepsilon,\delta}
		+\delta\,\displaystyle\int_{\Omega}
		\frac{\theta^{\varepsilon,\delta}(x)-\theta^{\varepsilon,\delta}(y)}{|x-y|^{d+2s}} \di y
		=\kappa\displaystyle\int_\Omega \psi_\varepsilon(x,y)\sin\big(\theta^{\varepsilon,\delta}(y)-\theta^{\varepsilon,\delta}(x)\big) \di y,
		\\[3ex]
		\theta^{\varepsilon,\delta} \Big|_{t = 0}=\theta^{\mathrm{in}}.
	\end{cases}
\end{equation}
Since $\psi_\varepsilon\in L^1(\Omega\times\Omega)$ and the sine function is globally Lipschitz,
the right-hand side of \eqref{C-2} is globally Lipschitz with respect to the $L^2$-norm (and $L^\infty$-norm).
Moreover, $-\delta(-\Delta)^s$ generates an analytic semigroup on $L^2(\Omega)$.
Hence \eqref{C-2} admits a unique solution on $[0,T]$. For more details on the existence of solutions to \eqref{C-2}, we refer to Proposition \ref{PA.1}. With the above preparations, we have the following diameter estimates.
\begin{lemma}[$L^\infty$-bound and bounded diameter]\label{L3.1}
Suppose that parameters satisfy 
\[s\in(0,1), \quad \varepsilon>0, \quad \delta>0, \]
and let $\theta^{\varepsilon,\delta}$ be a smooth solution to \eqref{C-2}
	with initial datum satisfying
	\begin{equation} \label{C-2-1}
	\theta^{\mathrm{in}}\in L^\infty(\Omega)\cap H^s(\Omega), \quad 
	{\mathcal D}[\theta^{\mathrm{in}}] <\pi.
	\end{equation}
	Then, one has the following estimates:~for $t \geq 0$, 
	\begin{equation} \label{C-3}
	\operatorname*{ess\,sup}_{x\in \Omega} \theta^{\varepsilon,\delta}(t)\le \operatorname*{ess\,sup}_{x\in \Omega} \theta^{\mathrm{in}}, \quad 
		\operatorname*{ess\,inf}_{x\in \Omega} \theta^{\varepsilon,\delta}(t)\ge \operatorname*{ess\,inf}_{x\in \Omega} \theta^{\mathrm{in}}, \quad 
	{\mathcal D}[\theta^{\varepsilon,\delta}(t)] \le {\mathcal D}[\theta^{\mathrm{in}}]<\pi.
	\end{equation}
\end{lemma}
\begin{proof} For the simplicity of notation, we denote $\theta = \theta^{\varepsilon,\delta}$ as long as there is no confusion. The proof relies on the truncation method combined with the continuity argument. \newline

We first fix $k\in\mathbb R$ and define 
\[ w:=(\theta-k)_+, \]
multiply \eqref{C-2} by $w$ and integrate the resulting relation over $\Omega$ to obtain
	\begin{equation} \label{C-4}
		\frac12\frac{\di}{\di t}\|w(t)\|_{L^2(\Omega)}^2
		+\delta\,\mathcal{A}(\theta(t),w(t))
		=\kappa\iint_{\Omega\times\Omega}\psi_\varepsilon(x,y)\sin(\theta(y)-\theta(x))\,w(x) \di y \di x,
	\end{equation}
	where $\mathcal{A}(u,v)$ is the symmetric bilinear form associated with the regional fractional Laplacian:
	\[
	\mathcal{A}(u,v):=\frac12\iint_{\Omega\times\Omega}\frac{(u(x)-u(y))(v(x)-v(y))}{|x-y|^{d+2s}} \di x \di y.
	\]
Since the map $r\mapsto (r-k)_+$ is nondecreasing, it holds pointwise
	\[
	(\theta(x)-\theta(y))(w(x)-w(y))\ge0.
	\]
	Hence, we have
	\begin{equation} \label{C-5}
	\mathcal{A}(\theta,w)\ge0. 
	\end{equation}
	We combine \eqref{C-4} and \eqref{C-5} to get
	\begin{equation}\label{C-6}
		\frac12\frac{\di}{\di t}\|w(t)\|_{L^2(\Omega)}^2
		\le
		\kappa\iint_{\Omega\times\Omega}\psi_\varepsilon(x,y)\sin(\theta(y)-\theta(x))\,w(x) \di y \di x,
	\end{equation}
	and we use
	\[ \psi_\varepsilon(x,y)=\psi_\varepsilon(y,x) \quad \mbox{and} \quad \sin(\theta(y)-\theta(x))=-\sin(\theta(x)-\theta(y)) \]
	to obtain
	\begin{align}\label{C-7} 
	\begin{split}
		&\iint_{\Omega\times\Omega} \psi_\varepsilon(x,y)\sin(\theta(y)-\theta(x))\,w(x) \di y \di x \\		
	&\hspace{1cm}	=\frac12\iint_{\Omega\times\Omega} \psi_\varepsilon(x,y)\sin(\theta(y)-\theta(x))\,(w(x)-w(y)) \di y \di x.
	\end{split}
	\end{align}
Next, we define
	\[
	T_*:=\sup\Big\{t>0:\ {\mathcal D}[\theta(\tau)] <\frac{\pi+ {\mathcal D}[\theta^{\mathrm{in}}]} {2}\ \text{for all }\tau\in[0,t)\Big\}.
	\]
By Proposition \ref{PA.1},
	the map $t\mapsto {\mathcal D}[\theta(t)]$ is continuous.
	 By assumption \eqref{C-2-1} and continuity, we have 
	 \[ T_*>0. \]	 
We claim that 
\begin{equation} \label{C-7-1}
T_* = \infty.
\end{equation}	
{\it Proof of \eqref{C-7-1}}: Suppose the contrary holds. Then, one has 
\begin{equation} \label{C-7-2}
{\mathcal D}[\theta(T_*)] = \frac{\pi+{\mathcal D}[\theta^{\mathrm{in}}]}{2}.
\end{equation}
For $t\in[0,T_*)$ we have 
	\[ |\theta(y)-\theta(x)|<\pi \quad \mbox{a.e.}, \]
	and thus we have
	\[
	\sin(\theta(y)-\theta(x))=(\theta(y)-\theta(x))\,m(\theta(y)-\theta(x)),
	\quad m(z):=\frac{\sin z}{z} \ge0 \ \mathrm{on} \ [-\pi,\pi].
	\]
	Since $(\theta(x)-\theta(y))(w(x)-w(y))\ge0$, we deduce
	\[
	\sin (\theta(y)-\theta(x))(w(x)-w(y)) = (\theta(y)-\theta(x))\,m(\theta(y)-\theta(x))  (w(x)-w(y))  \le0.
	\]
	Therefore, we use \eqref{C-7} and $\psi_\varepsilon\ge0$ to see that the right-hand side of \eqref{C-6} is nonpositive. Hence
	\[
	\frac{\di}{\di t}\|w(t)\|_{L^2(\Omega)}^2\le0,
	\qquad t\in[0,T_*).
	\]
In the sequel, we show that the estimates in \eqref{C-3} hold in the time-interval $[0, T^*)$. \newline	
	
\noindent $\bullet$~(Derivation of the first estimate in \eqref{C-3} in the time interval $[0, T^*)$):~We choose 
\[  k=\operatorname*{ess\,sup}_{x\in \Omega}\theta^{\mathrm{in}} \]
to get 
\[ w(0)=(\theta^{\mathrm{in}}-k)_+\equiv0 \quad \mbox{a.e.} \]
Thus we have
\[ \|w(t)\|_{L^2(\Omega)}^2=0 \quad \mbox{for all $t\in[0,T_*)$}. \]
This implies
	\begin{align} \label{C-8}
\operatorname*{ess\,sup}_{x\in \Omega}\theta(t)\le \operatorname*{ess\,sup}_{x\in \Omega}\theta^{\mathrm{in}},\qquad t\in[0,T_*).
\end{align}
\vspace{0.1cm}

\noindent $\bullet$~(Derivation of the second estimate in \eqref{C-3} in the time interval $[0, T^*)$):~We apply the same argument to $(k-\theta)_+$ with $k=\operatorname*{ess\,inf}_{x\in \Omega}\theta^{\mathrm{in}}$  to get 
	\begin{align} \label{C-9}
	\operatorname*{ess\,inf}_{x\in \Omega}\theta(t) \ge \operatorname*{ess\,inf}_{x\in \Omega}\theta^{\mathrm{in}},\qquad t\in[0,T_*).
	\end{align}
\vspace{0.1cm}

\noindent$\bullet$~(Derivation of the third estimate in \eqref{C-3} in the time interval $[0, T^*)$):~We combine \eqref{C-8} and \eqref{C-9} to get 	
	\begin{equation} \label{C-9-1}
	{\mathcal D}[\theta(t)] \le {\mathcal D}[\theta^{\mathrm{in}}] <\pi,
	\qquad t\in[0,T_*).
	\end{equation}
Letting $t \to T_*^-$, we have
\begin{equation} \label{C-9-2}
{\mathcal D}[\theta(T_*^{-})] \leq  {\mathcal D}[\theta^{\mathrm{in}}]  < \pi.
\end{equation}	
Then, we use $\eqref{C-2-1}_2$ and \eqref{C-9-2} to get 
	\[
	{\mathcal D}[\theta(T_*)] \leq  {\mathcal D}[\theta^{\mathrm{in}}]  < \frac{\pi+{\mathcal D}[\theta^{\mathrm{in}}]}{2}.
	\]
This is contradictory to \eqref{C-7-2}. Then, we verify \eqref{C-7-1}, i.e., $T_*=\infty$ and the desired estimates \eqref{C-8}, \eqref{C-9} and \eqref{C-9-1} hold for all $t \geq 0.$
\end{proof}
\vspace{0.25cm}

\subsection{Energy estimates}\label{sec:3.2}
In this subsection, we derive the energy identity for the approximate solution $\theta^{\varepsilon,\delta}$ in Section \ref{sec:3.1}. In particular, we establish uniform estimates for $\theta^{\varepsilon,\delta}$ independent of $\varepsilon$ (for fixed $\delta$).
\\[5pt]
For this, we multiply \eqref{C-2} by $\theta^{\varepsilon,\delta}$ and then integrate the resulting relation over $\Omega$ to get 
\begin{align*}
	\begin{aligned}
	&\frac12\frac{\di}{\di t}\|\theta^{\varepsilon,\delta}(t)\|_{L^2(\Omega)}^2
	+\delta\,\mathcal{A}\big(\theta^{\varepsilon,\delta}(t),\theta^{\varepsilon,\delta}(t)\big)\\
	&\qquad \qquad =\kappa\int_\Omega \theta^{\varepsilon,\delta}(t,x)\int_\Omega \psi_\varepsilon(x,y)
	\sin\big(\theta^{\varepsilon,\delta}(t,y)-\theta^{\varepsilon,\delta}(t,x)\big)\di y \di x,
	\end{aligned}
\end{align*}
where $\mathcal{A}$ is defined as in \eqref{B-3}.
\smallskip

\begin{lemma}[Energy dissipation identity]\label{L3.2}
	Let $\theta^{\varepsilon,\delta}$ be a global solution to the regularized problem \eqref{C-2}.  
	Then the following assertions hold.
\begin{enumerate}
\item
The total energy functional ${\mathcal E}_\varepsilon[\theta^{\varepsilon, \delta}]$ in $\eqref{B-3-1}_3$ satisfies dissipation estimate:
	\begin{equation}\label{C-10}
	 {\mathcal E}_\varepsilon[ \theta^{\varepsilon,\delta}(t)] + \int_0^t \|\partial_{\tau}\theta^{\varepsilon,\delta}(\tau)\|_{L^2(\Omega)}^2 \di \tau =  {\mathcal E}_\varepsilon[ \theta^{\mathrm{in}}], 
		\quad \text{for a.e. } t>0.
	\end{equation}
\item
The set $\{ \|\theta^{\varepsilon, \delta} (t)\|_{\dot H^s(\Omega)} \}$ is uniformly bounded in the parameter $\varepsilon$:
\begin{align} \label{NewC-16}
\|\theta^{\varepsilon,\delta} (t)\|_{\dot H^s(\Omega)}^2 \leq  \frac{\kappa+\delta}{\delta} \|\theta^{\rm in} \|_{\dot{H}_s(\Omega)}^2, \quad t > 0.
\end{align}	

\item
The set $\{ \| \partial_t \theta^{\varepsilon,\delta} (t) \|_{(H^s(\Omega))^*} \}$ is uniformly bounded in the parameter $\varepsilon$:
\begin{align} \label{NewC-14}
	\| \partial_t \theta^{\varepsilon,\delta} (t) \|_{(H^s(\Omega))^*}\leq \frac{\kappa +\delta}{2} \sqrt{\frac{\kappa+\delta}{\delta} } \|\theta^{\rm in} \|_{\dot{H}_s(\Omega)}, \quad t > 0.
\end{align}	
\end{enumerate}	

\end{lemma}
\begin{proof}
 For notational simplicity, we suppress $\delta$-dependence in $\theta^{\varepsilon, \delta}$, since we consider a fixed $\delta$:
 \[  \theta^\varepsilon \equiv \theta^{\varepsilon,\delta}. \]
 (1)~We differentiate ${\mathcal E}_{P, \varepsilon}[\theta^\varepsilon(t)]$ in $\eqref{B-3-1}_1$ with respect to $t$ to find 
\begin{align}
\begin{aligned} \label{C-10-1}
	\frac{\di}{\di t} {\mathcal E}_{P, \varepsilon}[\theta^\varepsilon(t)]
	&=\frac{\kappa}{2}\iint_{\Omega\times\Omega} \psi_\varepsilon(x,y)\,
	\sin(\theta^\varepsilon(x)-\theta^\varepsilon(y))\,
	(\partial_t\theta^\varepsilon(x)-\partial_t\theta^\varepsilon(y)) \di x \di y.
\end{aligned}
\end{align}
Using the symmetry $\psi_\varepsilon(x,y)=\psi_\varepsilon(y,x)$ and exchanging
$x$ and $y$ on the right-hand side of \eqref{C-10-1}, we obtain
\begin{align}
\begin{aligned} \label{C-10-2}
	\frac{\di}{\di t} {\mathcal E}_{P, \varepsilon}[\theta^\varepsilon(t)]
	&=\int_\Omega \partial_t\theta^\varepsilon(x)
	\left[\kappa\int_\Omega \psi_\varepsilon(x,y)\sin(\theta^\varepsilon(x)-\theta^\varepsilon(y)) \di y\right]\di x.
\end{aligned}
\end{align}
Similarly, we have 
\begin{align} \label{C-10-3}
	\begin{aligned}
&	\frac{\di}{\di t} {\mathcal E}_{K}[\theta^\varepsilon(t)]
	&=\int_\Omega \partial_t\theta^\varepsilon(x)
	\left[\delta\int_\Omega \psi(x,y)(\theta^\varepsilon(x)-\theta^\varepsilon(y)) \di y\right]\di x.
	\end{aligned}
\end{align}
On the other hand, we use $\sin(\theta^\varepsilon(x)-\theta^\varepsilon(y))=-\sin(\theta^\varepsilon(y)-\theta^\varepsilon(x))$ to see
\begin{equation} \label{C-10-4}
\partial_t\theta^\varepsilon(x)
= -\kappa \int_\Omega \psi_\varepsilon(x,y)\sin(\theta^\varepsilon(x)-\theta^\varepsilon(y)) \di y-\delta\int_\Omega \psi(x,y)(\theta^\varepsilon(x)-\theta^\varepsilon(y))\di y.
\end{equation}
We combine \eqref{C-10-2}, \eqref{C-10-3} and \eqref{C-10-4} to obtain 
\[
\frac{\di}{\di t} {\mathcal E}_\varepsilon[ \theta^{\varepsilon,\delta}(t)] = -\int_\Omega |\partial_t\theta^\varepsilon(x)|^2 \di x
=-\|\partial_t\theta^\varepsilon(t)\|_{L^2(\Omega)}^2.
\]
Finally, we integrate above equality over $(0,t)$ to get the desired energy estimate \eqref{C-10}:
\begin{equation}\label{C-11}
	{\mathcal E}_{P,\varepsilon}[\theta^\varepsilon(t)]+ {\mathcal E}_{K}[\theta^\varepsilon(t)]
	+\int_0^t \|\partial_\tau\theta^\varepsilon(\tau)\|_{L^2(\Omega)}^2 \di \tau
	=\mathcal E_{P,\varepsilon}(\theta^{\rm in})+ \mathcal E_{K}(\theta^{\rm in}), \quad t\ge0.
\end{equation}

\noindent (2)~For the desired uniform bound, we take two steps as follows. \newline

\noindent $\bullet$~\textbf{Step A} (Derivation of $ {\mathcal E}_\varepsilon[ \theta^{\varepsilon,\delta}(t)] \lesssim \|\theta^{\rm in} \|_{\dot{H}_s(\Omega)}^2$):~By direct calculations, one has 
\begin{align} 
\begin{aligned} \label{C-12}
\mathcal E_{P, \varepsilon}[\theta^{\rm in}] &= \frac{\kappa}{2}\iint_{\Omega\times \Omega}\frac{1 - \cos(\theta^{\rm in} (y) - \theta^{\rm in} (x))}{(|x-y|+\varepsilon)^{d+2s}}  \di y \di x \\
& = \frac{\kappa}{2} \iint_{\Omega \times \Omega}\frac{2\sin^2\frac{\theta^{\rm in} (y) - \theta^{\rm in} (x)}{2}}{(|x-y|+\varepsilon)^{d+2s}} \di y \di x \\
& \leq  \frac{\kappa}{4} \iint_{\Omega \times \Omega}\frac{(\theta^{\rm in} (y) - \theta^{\rm in} (x))^2}{(|x-y|+\varepsilon)^{d+2s}} \di y \di x \\
& \leq \frac{\kappa}{4} \|\theta^{\rm in} \|^2_{\dot{H}_s(\Omega)},
\end{aligned}
\end{align}
and
\begin{equation}\label{C-13}
{\mathcal E}_{K}[\theta^{\rm in}] = \frac{\delta}{4}\iint_{\Omega\times \Omega}\frac{(\theta^{\rm in} (x) - \theta^{\rm in} (y))^2}{|x-y|^{d+2s}} \, \di y \di x  = \frac{\delta}{4} \|\theta^{\rm in} \|_{\dot{H}_s(\Omega)}^2.
\end{equation}
We combine estimates in \eqref{C-11}, \eqref{C-12} and \eqref{C-13} to find 
\begin{equation} \label{C-13-1}
{\mathcal E}_\varepsilon[ \theta^{\varepsilon,\delta}(t)]={\mathcal E}_{P, \varepsilon}[\theta^\varepsilon(t)] + {\mathcal E}_{K}[\theta^\varepsilon(t)] \leq {\mathcal E}_{P,\varepsilon}[\theta^{\rm in}] + {\mathcal E}_{K}[\theta^{\rm in}] \leq \frac{\kappa+\delta}{4} \|\theta^{\rm in} \|_{\dot{H}_s(\Omega)}^2.
\end{equation}
\vspace{0.2cm}

\noindent $\bullet$~\textbf{Step B} (Derivation of $\|\theta^\varepsilon (t)\|_{\dot H^s(\Omega)}^2 \lesssim   \mathcal{E}_K (\theta^\varepsilon (t))$):~
Now, we use Lemma \ref{L3.1} to find 
\begin{align}\label{C-15}
	\|\theta^{\varepsilon}(t)\|_{L^\infty(\Omega)}^2\le C.
	\end{align}
Since the domain $\Omega$ is bounded, this also yields a uniform bound for the $L^2$-norm:
\begin{align}\label{C-15-1}
	\|\theta^{\varepsilon}(t)\|_{L^2(\Omega)}^2\le C.
	\end{align}
We use $\eqref{B-3-1}_2$ and \eqref{C-13-1} to get 
\begin{equation} \label{C-16}
\|\theta^\varepsilon (t)\|_{\dot H^s(\Omega)}^2 = \frac{4}{\delta} \mathcal{E}_K[\theta^\varepsilon (t)] \leq  \frac{\kappa+\delta}{\delta} \|\theta^{\rm in} \|_{\dot{H}_s(\Omega)}^2.
\end{equation}	
\vspace{0.2cm}

\noindent (3) We use the Cauchy-Schwarz inequality, definition of  $\|\varphi\|_{\dot H^{s}(\Omega)}^2$ in \eqref{A-6}, $|\sin \theta| \leq |\theta|$ and \eqref{C-13-1} to get 
\begin{align} \label{C-14}
	&\| \partial_t \theta^\varepsilon (t) \|_{(H^s(\Omega))^*}  = \sup_{\varphi \in H^s(\Omega)\atop\|\varphi\|_{H^s(\Omega)}\le1} \int_\Omega \partial_t \theta^\varepsilon (x) \varphi (x)  \di x \nonumber\\
	& \hspace{0.5cm} =  \sup_{\varphi \in H^s(\Omega)\atop\|\varphi\|_{H^s(\Omega)}\le1} \Bigg( \kappa \iint_{\Omega \times \Omega} \frac{\sin(\theta^\varepsilon (y) - \theta^\varepsilon (x))}{(|x-y|+\varepsilon)^{d+2s}}  \varphi (x)  \di y \di x +\delta\iint_{\Omega \times \Omega} \frac{\theta^\varepsilon (y) - \theta^\varepsilon (x)}{|x-y|^{d+2s}}  \varphi (x)  \di y \di x \Bigg) \nonumber\\
	& \hspace{0.5cm} \leq  \sup_{\varphi \in H^s(\Omega)\atop\|\varphi\|_{H^s(\Omega)}\le1} \Bigg( \frac{\kappa}{2} \iint_{\Omega \times \Omega} \frac{\sin(\theta^\varepsilon (y) - \theta^\varepsilon (x))}{(|x-y|+\varepsilon)^{d+2s}}  (\varphi (x) - \varphi (y))  \di y \di x \nonumber\\
	&\hspace{5.5cm}+ \frac{\delta}{2} \iint_{\Omega \times \Omega}\frac{\theta^\varepsilon (y) - \theta^\varepsilon (x)}{|x-y|^{d+2s}}  (\varphi (x) - \varphi (y)) \di y \di x \Bigg) \nonumber\\
	& \hspace{0.5cm} \leq  \sup_{\varphi \in H^s(\Omega)\atop\|\varphi\|_{H^s(\Omega)}\le1} \| \varphi \|_{H^s(\Omega)}  \Bigg( \frac{\kappa}{2} \left( \iint_{\Omega \times \Omega} \frac{\sin^2 (\theta^\varepsilon (y) - \theta^\varepsilon (x))}{(|x-y|+\varepsilon)^{d+2s}}  \di y \di x \right)^{\frac{1}{2}} \nonumber\\
	& \hspace{6.5cm} + \frac{\delta}{2} \left( \iint_{\Omega \times \Omega}\frac{(\theta^\varepsilon (y) - \theta^\varepsilon (x))^2}{|x-y|^{d+2s}}  \di y \di x \right)^{\frac{1}{2}} \Bigg) \\  
	& \hspace{0.5cm} \leq  \frac{\kappa}{2} \left( \frac{4}{\delta}\mathcal{E}_K [\theta^\varepsilon (t)] \right)^\frac{1}{2}  + \frac{\delta}{2}\left(\frac{4}{\delta} \mathcal{E}_K[\theta^\varepsilon (t)] \right)^\frac{1}{2}\nonumber \\
	& \hspace{0.5cm} \leq  \frac{\kappa}{2} \left( \frac{4}{\delta}\frac{\kappa+\delta}{4} \|\theta^{\rm in} \|_{\dot{H}_s(\Omega)}^2 \right)^\frac{1}{2}  + \frac{\delta}{2}\left(\frac{4}{\delta} \frac{\kappa+\delta}{4} \|\theta^{\rm in} \|_{\dot{H}_s(\Omega)}^2 \right)^\frac{1}{2}   \nonumber\\
	& \hspace{0.5cm} \leq \frac{\kappa +\delta}{2} \sqrt{\frac{\kappa+\delta}{\delta} } \|\theta^{\rm in} \|_{\dot{H}_s(\Omega)}.\nonumber
\end{align}
\end{proof}
With the above estimates \eqref{NewC-16} and \eqref{NewC-14}, we are ready to use Aubin--Lions lemma (see Lemma \ref{L2.1}) in next subsection.
\subsection{Passage to the limit in $\{ \theta^{\varepsilon, \delta} \}_{\varepsilon}$} We use the embeddings
\[
H^s(\Omega)\hookrightarrow \hookrightarrow  L^2(\Omega)\hookrightarrow (H^{s}(\Omega))^\ast,
\]
where the first embedding is compact and the second embedding is continuous for $\Omega \subset \bbr^d, d\ge1$ bounded with Lipschitz boundary \cite{DiNezza2012}. By \eqref{C-15-1}, \eqref{C-16}, \eqref{C-14}, Aubin--Lions lemma \ref{L2.1} and Banach–Alaoglu theorem,
there exists a subsequence $\varepsilon_j\to0$ and a limit function
\begin{equation} \label{C-16-1}
\theta^\delta\in L^2([0,T);H^s(\Omega))\cap C([0,T);L^2(\Omega))
\end{equation}
such that
\begin{equation}\label{C-17}
\begin{cases}
\displaystyle \theta^{\varepsilon_j,\delta}\to \theta^\delta\quad\text{strongly in }C([0,T);L^2(\Omega)), \\[1em]
\displaystyle \theta^{\varepsilon_j,\delta}\rightharpoonup \theta^\delta\quad\text{weakly in }L^2([0,T);H^s(\Omega)).
\end{cases}
\end{equation}
Moreover, by Lemma \ref{L3.1} and subsequence extraction, we have
for all $t\ge0$,
\begin{equation*}
	\operatorname*{ess\,sup}_{x\in \Omega} \theta^{\delta}(t,\cdot)\le \operatorname*{ess\,sup}_{x\in \Omega} \theta^{\mathrm{in}},\qquad
	\operatorname*{ess\,inf}_{x\in \Omega} \theta^{\delta}(t,\cdot)\ge \operatorname*{ess\,inf}_{x\in \Omega} \theta^{\mathrm{in}}.
\end{equation*}
In particular, it follows from Proposition  \ref{PA.2} that 
\begin{equation*}
	{\mathcal D}[\theta^{\delta}(t)] \le {\mathcal D}[\theta^{\mathrm{in}}] <\pi,
	\qquad t\ge0.
\end{equation*}
Next, we show that the limit $\theta^{\delta}$ satisfies the following Cauchy problem:
\begin{align}\label{C-18}
\begin{cases}
\displaystyle \partial_t\theta^\delta(t,x)
	+\delta\,\displaystyle\int_{\Omega}\frac{\theta^\delta(t,x)-\theta^\delta(t,y)}{|x-y|^{d+2s}}\,\di y \\[2em]
\displaystyle \hspace{1.5cm} =\kappa\int_\Omega \frac{\sin\big(\theta^\delta(t,y)-\theta^\delta(t,x)\big)}{|x-y|^{d+2s}}\,\di y,~~t > 0,~x \in \Omega, \\[2em]
\displaystyle \theta^\delta \Big|_{t = 0} =\theta^{\mathrm{in}}.
\end{cases}
\end{align}
\begin{lemma}\label{L3.3}
The limit $\theta^\delta$ satisfying the regularity \eqref{C-16-1} is a weak solution to \eqref{C-18}.
\end{lemma}
\begin{proof} We first fix $T>0$ and $\delta>0$, and let $\{\theta^{\varepsilon,\delta}\}_{\varepsilon>0}$ be a family of solutions to
	\eqref{C-2}. By \eqref{C-15-1}--\eqref{C-14}, there exists a constant independent of $\varepsilon$ such that
	\begin{equation}\label{C-19}
		\|\theta^{\varepsilon,\delta}\|_{L^\infty([0,T);L^2(\Omega))}
		+\|\theta^{\varepsilon,\delta}\|_{L^2([0,T);H^s(\Omega))}
		+\|\partial_t\theta^{\varepsilon,\delta}\|_{L^2([0,T);(H^{s}(\Omega))^\ast)}
		\le C(T, \delta,\kappa).
	\end{equation}
	As discussed above in \eqref{C-17}, there exists $\varepsilon_j\to0$ and
	\[
	\theta^\delta\in L^2([0,T);H^s(\Omega))\cap C([0,T);L^2(\Omega))
	\]
	such that
	\begin{align}
	\begin{aligned} \label{C-20}
	& \theta^{\varepsilon_j,\delta}\to \theta^\delta \quad\text{strongly in } C([0,T);L^2(\Omega)), \\
        & \theta^{\varepsilon_j,\delta}\rightharpoonup \theta^\delta \quad\text{weakly in } L^2([0,T);H^s(\Omega)).
	\end{aligned}
	\end{align}
	Since the proof is rather lengthy, we split the proof into four steps.
	
	\medskip
	\noindent$\bullet $ \textbf{Step A} (Weak formulation and symmetrization):
	Let $\varphi\in C_c^\infty([0,T)\times\Omega)$. We multiply \eqref{C-2} by $\varphi$ and integrate the resulting relation over
	$[0,T)\times\Omega$, and then we use integration by parts in time to obtain
	\begin{align}
	\begin{aligned} \label{C-21}
		&-\int_0^T\!\!\int_\Omega \theta^{\varepsilon,\delta}\,\partial_t\varphi \di x\di t
		-\int_\Omega \theta^{\rm in}(x)\,\varphi(0,x)\,\di x
		+\delta\int_0^T \mathcal{A}(\theta^{\varepsilon,\delta}(t),\varphi(t))\di t \\
		&\hspace{1cm} = \kappa\int_0^T\!\!\iint_{\Omega \times \Omega}
		\psi_\varepsilon(x,y)\,\sin\!\big(\theta^{\varepsilon,\delta}(t,y)-\theta^{\varepsilon,\delta}(t,x)\big)\,
		\varphi(t,x) \di y\di x\di t, 
	\end{aligned}
	\end{align}
	where $\mathcal{A}$ is defined as in \eqref{B-3}. Using the oddness of the sine function and exchanging $x$ and $y$, the right-hand side of \eqref{C-21} can be symmetrized as
	\begin{align*}
		&\iint_{\Omega \times \Omega} \psi_\varepsilon(x,y)\sin(\theta^{\varepsilon, \delta}(y)-\theta^{\varepsilon, \delta}(x))\,\varphi(x)\di x\di y \\
		&\hspace{1cm}=\frac12\iint_{\Omega \times \Omega} \psi_\varepsilon(x,y)\sin(\theta^{\varepsilon, \delta}(y)-\theta^{\varepsilon, \delta}(x))(\varphi(x)-\varphi(y)) \di x\di y .
	\end{align*}
	Hence the relation \eqref{C-21} is equivalent to
	\begin{align}
		&-\int_0^T\!\!\int_\Omega \theta^{\varepsilon,\delta}\,\partial_t\varphi \di x\di t
		-\int_\Omega \theta^{\rm in}(x)\,\varphi(0,x)\,\di x
		+\delta\int_0^T \mathcal{A}(\theta^{\varepsilon,\delta}(t),\varphi(t)) \di t \notag\\
		&\qquad = \frac{\kappa}{2}\int_0^T\!\!\iint_{\Omega\times\Omega}
		\psi_\varepsilon(x,y)\,\sin\!\big(\theta^{\varepsilon,\delta}(t,y)-\theta^{\varepsilon,\delta}(t,x)\big)\,
		(\varphi(t,x)-\varphi(t,y)) \di y\di x\di t. \label{C-22}
	\end{align}
	
	\medskip
	\noindent$\bullet $ \textbf{Step B} (Passage to the limit on the left-hand side):~By the strong convergence in \eqref{C-20}, we obtain
	\[
	\int_0^T\!\!\int_\Omega \theta^{\varepsilon_j,\delta}\,\partial_t\varphi \di x\di t
	\to
	\int_0^T\!\!\int_\Omega \theta^\delta\,\partial_t\varphi \di x\di t.
	\]
	Since $\varphi(\cdot,t)\in H^s(\Omega)$, the map
	\[
	u\mapsto  \mathcal{A}(u,\varphi(t))
	\]
	defines a continuous linear functional on $H^s(\Omega)$ and
	\[
	|\mathcal A(u,\varphi(t))|\le \frac12 \|u\|_{H^s(\Omega)}\|\varphi(t)\|_{H^s(\Omega)}.
	\]
	Hence $t\mapsto \mathcal (u\mapsto \mathcal{A}(u,\varphi(t)))$ belongs to $L^2([0,T);(H^s(\Omega))^*)$.
	Since $\theta^{\varepsilon_j,\delta}\rightharpoonup \theta^\delta$ weakly in
	$L^2([0,T);H^s(\Omega))$ in \eqref{C-20}, by the characterization of weak convergence in Bochner spaces we obtain
	\[
	\int_0^T \mathcal A(\theta^{\varepsilon_j,\delta}(t),\varphi(t))\,\di t
	\to
	\int_0^T \mathcal A(\theta^\delta(t),\varphi(t))\,\di t.
	\]
	
	\medskip
	\noindent$\bullet $ \textbf{Step C} (Passage to the limit in the nonlocal nonlinear term):~We set 
	\[
	R_{\varepsilon}(t):=\iint_{\Omega\times \Omega}
	\psi_\varepsilon(x,y)\,\sin\!\big(\theta^{\varepsilon,\delta}(t,y)-\theta^{\varepsilon,\delta}(t,x)\big)\,
	(\varphi(t,x)-\varphi(t,y))\,\di x\di y \quad  \mathrm{for} \  \mathrm{a.e.} \ t\in(0,T).
	\]
	We claim that along the subsequence $\varepsilon_j\to0$, it holds
	\begin{equation}\label{C-23}
		\int_0^T R_{\varepsilon_j}(t)\,\di t
		\to
		\int_0^T\!\!\iint_{\Omega \times \Omega}
		\frac{\sin\!\big(\theta^\delta(t,y)-\theta^\delta(t,x)\big)}{|x-y|^{d+2s}}\,
		(\varphi(t,x)-\varphi(t,y))\,\di x\di y\,\di t .
	\end{equation}
	{\it Proof of \eqref{C-23}}:  Fix $\rho\in(0,1)$ and we decompose $\Omega \times \Omega$ as 
	\[
	\Omega\times\Omega = \Big \{(x,y):|x-y|>\rho\}\cup \{(x,y):|x-y|\le\rho \Big \},
	\]
	and we write 
	\begin{align*}
	\begin{aligned}
	R_{\varepsilon}(t) &=\iint_{|x-y|>\rho}
	\psi_\varepsilon(x,y)\,\sin\!\big(\theta^{\varepsilon,\delta}(t,y)-\theta^{\varepsilon,\delta}(t,x)\big)\,
	(\varphi(t,x)-\varphi(t,y))\,\di x\di y \\
	&\qquad+ \iint_{|x-y| \leq \rho}
	\psi_\varepsilon(x,y)\,\sin\!\big(\theta^{\varepsilon,\delta}(t,y)-\theta^{\varepsilon,\delta}(t,x)\big)\,
	(\varphi(t,x)-\varphi(t,y))\,\di x\di y  \\
	&=: R_{\varepsilon}^{>\rho}(t)+R_{\varepsilon}^{\le\rho}(t).
	\end{aligned}
	\end{align*}
	We further consider two cases. \newline
	
	\noindent$\diamond$ \textbf{Case C.1}: On $\{|x-y|>\rho\}$, we have 
	\[
	0\le \psi_{\varepsilon_j}(x,y)\le \rho^{-d-2s},
	\qquad
	\psi_{\varepsilon_j}(x,y)\to |x-y|^{-d-2s}\quad\text{uniformly on }\{|x-y|\ge\rho\}.
	\]
	Moreover, $|\sin(\cdot)|\le 1$ and $\varphi$ is bounded, hence the integrands in $R_{\varepsilon_j} (t)$ are dominated by an $L^1$-function
	(independent of $j$) on $\{|x-y|>\rho\}$. Since $\theta^{\varepsilon_j,\delta}\to\theta^\delta$ strongly in
	$C([0,T);L^2(\Omega))$, up to a further subsequence, we also have
	\[ \theta^{\varepsilon_j,\delta}(t,x)\to\theta^\delta(t,x) \quad \mbox{for a.e.\ $(t,x)$}.  \]
	This implies
	\[
	\sin\big(\theta^{\varepsilon_j,\delta}(t,y)-\theta^{\varepsilon_j,\delta}(t,x)\big)
	\to
	\sin\big(\theta^\delta(t,y)-\theta^\delta(t,x)\big)
	\quad\text{for a.e.\ }(t,x,y).
	\]
	Therefore, by dominated convergence theorem (for fixed $\rho$), one has 
	\begin{equation}\label{C-24}
		\int_0^T R_{\varepsilon_j}^{>\rho}(t) \di t
		\to
		\int_0^T\!\!\iint_{\{|x-y|>\rho\}}
		\frac{\sin\!\big(\theta^\delta(t,y)-\theta^\delta(t,x)\big)}{|x-y|^{d+2s}}\,
		(\varphi(t,x)-\varphi(t,y)) \di x\di y \di t .
	\end{equation}
	\vspace{0.2cm}
	
	\noindent$\diamond$ \textbf{Case C.2}: On $\{ |x-y|\le\rho \}$, we find an estimate uniform in $\varepsilon$.
	Using the smoothness of $\varphi$, we obtain
	\[
	|\varphi(t,x)-\varphi(t,y)|\le \|\nabla\varphi\|_{L^\infty((0,T)\times\Omega)}\,|x-y|.
	\]
	Moreover, it holds 
	\[ \psi_{\varepsilon}(x,y)\le |x-y|^{-d-2s} \quad \mbox{and} \quad |\sin a|\le |a|. \]
	Hence, for a.e.\ $t$,
	\begin{align}
		|R_{\varepsilon}^{\le\rho}(t)|
		&\le \iint_{\{|x-y|\le\rho\}}
		\frac{|\theta^{\varepsilon,\delta}(t,y)-\theta^{\varepsilon,\delta}(t,x)|}{|x-y|^{d+2s}}\,
		|\varphi(t,x)-\varphi(t,y)| \di x\di y \notag\\
		&\le\|\nabla\varphi\|_{L^\infty((0,T)\times\Omega)} \iint_{\{|x-y|\le\rho\}}
		\frac{|\theta^{\varepsilon,\delta}(t,y)-\theta^{\varepsilon,\delta}(t,x)|}{|x-y|^{d+2s-1}} \di x\di y  \label{C-25} \\
		&\le\|\nabla\varphi\|_{L^\infty((0,T)\times\Omega)}
		\| \theta^{\varepsilon,\delta}(t) \|_{\dot{H}^s(\Omega)}	\left(\iint_{\{|x-y|\le\rho\}}\frac{1}{|x-y|^{d+2s-2}} \di x\di y\right)^{\!\!1/2}. \notag 
	\end{align}
	By \eqref{C-19}, the Gagliardo seminorm of $\theta^{\varepsilon,\delta}(t)$ is bounded by a constant independent of $\varepsilon$. For the last term, a change of variables $z=x-y$ gives
	\begin{align*}
	\begin{aligned}
	& \iint_{\{|x-y|\le\rho\}}\frac{1}{|x-y|^{d+2s-2}}\,\di x\di y \\
	& \hspace{1cm} \le C_\Omega \iint_{\{|z|\le\rho\}}\frac{1}{|z|^{d+2s-2}}\,dz = C_{\Omega,d} \int_0^\rho r^{1-2s}\,dr = \frac{C_{\Omega,d}}{2-2s}\,\rho^{2-2s} \quad \to 0,
	\end{aligned}
	\end{align*}
 as $\rho\to0$ since $s\in(0,1)$. We combine this with \eqref{C-25} and integrate the resulting relation in time to get the uniform smallness estimate:
	\begin{equation}\label{C-26}
		\sup_{\varepsilon>0}\int_0^T |R_{\varepsilon}^{\le\rho}(t)|\,\di t
		\le C\rho^{1-s}\qquad\text{for all }\rho\in(0,1),
	\end{equation}
	for a constant $C$ independent of $\varepsilon$ and $\rho$. \newline

	Let $j\to\infty$ in \eqref{C-24} and then let $\rho\to0$.
	Since the near-field contributions are uniformly small by \eqref{C-26}, we obtain \eqref{C-23}.
	
	\medskip
	\noindent$\bullet $ \textbf{Step D} (Limiting equation):~Passing to the limit $j\to\infty$ in \eqref{C-22} using \textbf{Step~B} and \eqref{C-23}, we conclude that for every
	$\varphi\in C_c^\infty([0,T)\times\Omega)$,
	\begin{align*}
		&-\int_0^T\!\!\int_\Omega \theta^{\delta}\,\partial_t\varphi\,\di x\di t
		-\int_\Omega \theta^{\rm in}(x)\,\varphi(0,x)\,\di x
		+\delta\int_0^T \mathcal A(\theta^{\delta}(t),\varphi(t))\,\di t \notag\\
		&\hspace{1cm} = \frac{\kappa}{2}\int_0^T\!\!\iint_{\Omega \times \Omega}
		\frac{\sin\!\big(\theta^{\delta}(t,y)-\theta^{\delta}(t,x)\big)}{|x-y|^{d+2s}}\,
		(\varphi(t,x)-\varphi(t,y))\,\di x\di y\,\di t. 
	\end{align*}
	That is, $\theta^\delta$ is a global weak solution to \eqref{C-18}.
\end{proof}
\subsection{Uniform estimates for $\{ \theta^ \delta \}_{\delta}$.}\label{sec:3.4}
\medskip
Recall that the limit $\theta^{\delta}$ satisfies
\begin{align} \label{C-27}
	\operatorname*{ess\,sup}_{x\in \Omega} \theta^{\delta}(t,\cdot)\le \operatorname*{ess\,sup}_{x\in \Omega} \theta^{\mathrm{in}},\qquad
	\operatorname*{ess\,inf}_{x\in \Omega} \theta^{\delta}(t,\cdot)\ge \operatorname*{ess\,inf}_{x\in \Omega} \theta^{\mathrm{in}},
\end{align}
and in particular
\begin{equation}\label{C-28}
	{\mathcal D}[\theta^{\delta}(t)] \le {\mathcal D}[\theta^{\mathrm{in}}] <\pi,
	\qquad t\ge0.
\end{equation}
In the sequel, we now establish the uniform estimates for $\{ \theta^\delta \}$ independent of $\delta$. \newline

Similar to \eqref{B-3-1}, we define:
\begin{equation} \label{C-28-1}
\begin{cases} 
\displaystyle {\mathcal E}_{P}[\theta^{\delta}] :=\frac{\kappa}{2}\iint_{\Omega\times\Omega}
\psi(x,y)\,\big(1-\cos(\theta^\delta(x)-\theta^\delta(y))\big) \di x \di y, \\[3ex]
\displaystyle \mathcal{E}_K[\theta^\delta] := \frac{\delta}{4} \iint_{\Omega \times \Omega} \psi(x,y) (\theta^\delta (x) - \theta^\delta (y))^2 \di x \di y, \\[3ex]
\displaystyle \mathcal{E}[\theta^\delta] : = \mathcal{E}_P[\theta^\delta] + \mathcal{E}_K[\theta^\delta].
\end{cases}
\end{equation}
Compared to previous subsection, we need to use potential energy ${\mathcal E}_{P}[\theta^{\delta}] $ and \eqref{C-28} to obtain the control of the ${H^s(\Omega)}$-norm.

\begin{lemma} \label{L3.4}
Let $\theta^\delta$ be a solution to \eqref{C-18} satisfying \eqref{C-16-1}. Then $\theta^\delta$ satisfies the energy identity:
\begin{equation} \label{C-29}
\mathcal{E}[\theta^\delta(t)] + \int_0^t \|\partial_{\tau}\theta^\delta(\tau)\|_{L^2(\Omega)}^2 \di \tau	= \mathcal{E}[\theta^{\mathrm{in}}], \quad t > 0. 
 \end{equation}
\end{lemma}
\begin{proof}
Basically, we use the same arguments as in the proof of Lemma \ref{L3.2}. For this, we differentiate ${\mathcal E}_{P}[\theta^\delta(t)]$ with respect to $t$ to get 
\[
	\frac{\di}{\di t} {\mathcal E}_{P}[\theta^\delta(t)]
	=\frac{\kappa}{2}\iint_{\Omega\times\Omega} \psi(x,y)\,
	\sin(\theta^\delta(x)-\theta^\delta(y))\,
	(\partial_t\theta^\delta(x)-\partial_t\theta^\delta(y)) \di x \di y.
\]
Using the symmetry $\psi(x,y)=\psi(y,x)$ and exchanging
$x$ and $y$ on the right-hand side, we obtain
\begin{align*}
	\frac{\di}{\di t} {\mathcal E}_{P}[\theta^\delta(t)]
	=\int_\Omega \partial_t\theta^\delta(x)
	\left[\kappa\int_\Omega \psi(x,y)\sin(\theta^\delta(x)-\theta^\delta(y)) \di y\right]\di x.
\end{align*}
Similarly, we have 
\[
	\frac{\di}{\di t} {\mathcal E}_{K} [\theta^\delta(t)]
	=\int_\Omega \partial_t\theta^\delta(x)
	\left[\delta\int_\Omega \psi(x,y)(\theta^\delta(x)-\theta^\delta(y)) \di y\right]\di x.
\]
Since 
\[ \sin(\theta^\delta(x)-\theta^\delta(y))=-\sin(\theta^\delta(y)-\theta^\delta(x)), \]
we have
\[
\partial_t\theta^\delta(x)
= -\kappa \int_\Omega \psi(x,y)\sin(\theta^\delta(x)-\theta^\delta(y))\,\di y-\delta\int_\Omega \psi(x,y)(\theta^\delta(x)-\theta^\delta(y))\,\di y.
\]
Therefore, we obtain 
\[
\frac{\di}{\di t} {\mathcal E}_{P}[\theta^\delta(t)] + \frac{\di}{\di t} {\mathcal E}_{K}[\theta^\delta(t)]
= -\int_\Omega |\partial_t\theta^\delta(x)|^2\,\di x =-\|\partial_t\theta^\delta(t)\|_{L^2(\Omega)}^2.
\]
We integrate the above equality over $(0,t)$ to find the desired estimate \eqref{C-29}:
\begin{align} \label{B.1}
	{\mathcal E}_{P}[\theta^\delta(t)] +{\mathcal E}_{K}[\theta^\delta(t)]
	+\int_0^t \|\partial_\tau\theta^\delta(\tau)\|_{L^2(\Omega)}^2\,\di \tau
	= {\mathcal E}_{P}[\theta^{\rm in}] + {\mathcal E}_{K}[\theta^{\rm in}], \quad t\ge0.
\end{align}
\end{proof}

\vspace{0.2cm}

\noindent Using the same argument as in \eqref{C-12} and \eqref{C-13}, we can derive the following estimates:
\begin{align*}
{\mathcal E}_{P}[\theta^{\rm in}] \leq  \frac{\kappa}{4} \|\theta^{\rm in} \|^2_{\dot{H}_s(\Omega)}, \quad {\mathcal E}_{K}[\theta^{\rm in}]  \leq  \frac{\delta}{4} \|\theta^{\rm in} \|_{\dot{H}_s(\Omega)}^2.
\end{align*}
Therefore, the relation \eqref{B.1} yields
\begin{align}\label{Newnew3.46}
{\mathcal E}_{P}[\theta^\delta(t)] + {\mathcal E}_{K}[\theta^\delta(t)] \leq {\mathcal E}_{P}[\theta^{\rm in}] + {\mathcal E}_{K} [\theta^{\rm in}] \leq \frac{\kappa+\delta}{4} \|\theta^{\rm in} \|_{\dot{H}_s(\Omega)}^2.
\end{align}

Next, we claim that 
\begin{equation} \label{C-29-1}
\|\theta^{\delta}(t) \|_{\dot{H}_s(\Omega)}^2=\iint_{\Omega \times \Omega }\frac{(\theta^{\delta}(t,y) - \theta^{\delta}(t,x))^2}{|x-y|^{d+2s}} \, \di y \di x \le C .
\end{equation}
{\it Proof of \eqref{C-29-1}}:~For ${\mathcal D}[\theta^{\mathrm{in}}]= M\in[0,\pi)$, and $x\in[0,M]$ we have 
\begin{equation} \label{C-29-2}
	\frac{x}{\sin x}\le \frac{M}{\sin M}.
\end{equation}
Next, we derive the following estimates one by one.
\begin{align}
\begin{aligned} \label{C-29-3}
& (i)~\iint_{\Omega \times \Omega} \frac{\sin^2 (\theta^\delta(y) - \theta^\delta(x))}{|x-y|^{d+2s}} \, \di y \di x  \leq  \frac{\kappa+\delta}{\kappa} \|\theta^\mathrm{in} \|_{\dot{H}^s(\Omega)}^2. \\
& (ii)~\iint_{\Omega \times \Omega} \frac{(\theta^{\delta}(t,y) - \theta^{\delta}(t,x))^2}{|x-y|^{d+2s}} \, \di y \di x \leq  \left(\frac{M}{\sin M}\right)^2 \frac{\kappa+\delta}{\kappa} \|\theta^\mathrm{in} \|_{\dot{H}^s(\Omega)}^2.
\end{aligned}
\end{align}
\noindent $\diamond$~Case A (Derivation of the first inequality in \eqref{C-29-3}): Now, we use  \eqref{Newnew3.46} to see
\begin{align*}
	\mathcal{E}_{P}[\theta^\delta (t)] &= \frac{\kappa}{2}\iint_{\Omega \times \Omega} \frac{1 - \cos(\theta^\delta (y) - \theta^\delta (x))}{|x-y|^{d+2s}} \, \di y \di x \\
	&= \kappa \iint_{\Omega \times \Omega} \frac{\sin^2\frac{\theta^\delta(y) - \theta^\delta (x)}{2}}{|x-y|^{d+2s}} \, \di y \di x \leq 	\frac{\kappa+\delta}{4} \|\theta^\mathrm{in} \|_{\dot{H}^s(\Omega)}^2.
\end{align*}
By direct calculations with the above estimates, one has 
\begin{align}
\begin{aligned} \label{C-29-4}
&\iint_{\Omega \times \Omega} \frac{\sin^2 (\theta^\delta(y) - \theta^\delta(x))}{|x-y|^{d+2s}} \, \di y \di x \\
& \hspace{1cm} = 4  \iint_{\Omega \times \Omega} \frac{\sin^2 \frac{\theta^\delta (y) - \theta^\delta (x)}{2} \cos^2 \frac{\theta^\delta (y) - \theta^\delta (x)}{2}}{|x-y|^{d+2s}} \, \di y \di x \\
& \hspace{1cm} \leq  4 \iint_{\Omega \times \Omega} \frac{\sin^2 \frac{\theta^\delta (y) - \theta^\delta (x)}{2} }{|x-y|^{d+2s}} \, \di y \di x =\frac{4}{\kappa}\mathcal{E}_{P}[\theta^\delta (t)] \leq 	\frac{\kappa+\delta}{\kappa} \|\theta^\mathrm{in} \|_{\dot{H}^s(\Omega)}^2.
\end{aligned}
\end{align}

\noindent $\diamond$~Case B (Derivation of the second inequality in \eqref{C-29-3}):~Again, we use \eqref{Newnew3.46} and \eqref{C-29-2} to obtain
\begin{align}\label{C-30}
	\begin{aligned}
		\|\theta^{\delta}(t) \|_{\dot{H}_s(\Omega)}^2&=\iint_{\Omega \times \Omega} \frac{(\theta^{\delta}(t,y) - \theta^{\delta}(t,x))^2}{|x-y|^{d+2s}} \, \di y \di x  \\
		& =	\iint_{\Omega \times \Omega} \frac{(\theta^{\delta}(t,y) - \theta^{\delta}(t,x))^2}{\sin^2 (\theta^{\delta}(t,y) - \theta^{\delta}(t,x))}\frac{\sin^2 (\theta^{\delta}(t,y) - \theta^{\delta}(t,x))}{|x-y|^{d+2s}} \, \di y \di x\\
		&\le  \left(\frac{M}{\sin M}\right)^2\iint_{\Omega \times \Omega} \frac{\sin^2 (\theta^{\delta}(t,y) - \theta^{\delta}(t,x))}{|x-y|^{d+2s}} \, \di y \di x\\
		& \le  \left(\frac{M}{\sin M}\right)^2 \frac{\kappa+\delta}{\kappa} \|\theta^\mathrm{in} \|_{\dot{H}^s(\Omega)}^2.
	\end{aligned}
\end{align}
This finishes the proof of \eqref{C-29-1}.

Next, we use the estimates in \eqref{C-29-3} to see that 
\begin{align} \label{C-31}
&\| \partial_t \theta^\delta (t) \|_{(H^s(\Omega))^*} = \sup_{\varphi \in H^s(\Omega)\atop\|\varphi\|_{H^s(\Omega)}\le1} \int_\Omega \partial_t \theta^\delta (x) \varphi (x)  \di x \nonumber\\
		&=  \sup_{\varphi \in H^s(\Omega)\atop\|\varphi\|_{H^s(\Omega)}\le1} \Bigg( \kappa \iint_{\Omega \times \Omega} \frac{\sin(\theta^\delta (y) - \theta^\delta (x))}{|x-y|^{d+2s}}  \varphi (x)  \di y \di x +\delta\iint_{\Omega \times \Omega} \frac{\theta^\delta (y) - \theta^\delta (x)}{|x-y|^{d+2s}}  \varphi (x)  \di y \di x \Bigg) \nonumber\\
		& =  \sup_{\varphi \in H^s(\Omega)\atop\|\varphi\|_{H^s(\Omega)}\le1} \Bigg( \frac{\kappa}{2} \iint_{\Omega \times \Omega} \frac{\sin(\theta^\delta (y) - \theta^\delta (x))}{|x-y|^{d+2s}}  (\varphi (x) - \varphi (y))  \di y \di x \nonumber\\
		 &\hspace{5.5cm}+ \frac{\delta}{2} \iint_{\Omega \times \Omega}\frac{\theta^\delta (y) - \theta^\delta (x)}{|x-y|^{d+2s}}  (\varphi (x) - \varphi (y)) \di y \di x \Bigg) \nonumber\\
		& \leq  \sup_{\varphi \in H^s(\Omega)\atop\|\varphi\|_{H^s(\Omega)}\le1} \| \varphi \|_{H^s(\Omega)}  \Bigg( \frac{\kappa}{2} \left( \iint_{\Omega \times \Omega} \frac{\sin^2 (\theta^\delta (y) - \theta^\delta(x))}{|x-y|^{d+2s}}  \di y \di x \right)^{\frac{1}{2}} \nonumber\\
		&\hspace{5.5cm}+ \frac{\delta}{2} \left( \iint_{\Omega \times \Omega}\frac{(\theta^\delta (y) - \theta^\delta (x))^2}{|x-y|^{d+2s}}  \di y \di x \right)^{\frac{1}{2}} \Bigg) \nonumber\\  
		& \leq  \frac{\kappa}{2} \left( \frac{\kappa+\delta}{\kappa} \|\theta^\mathrm{in} \|_{\dot{H}^s(\Omega)}^2 \right)^\frac{1}{2}  + \frac{\delta}{2} \frac{M}{\sin M} \left( \frac{\kappa+\delta}{\kappa} \|\theta^\mathrm{in} \|_{\dot{H}^s(\Omega)}^2 \right)^\frac{1}{2} \\
		& =  \left( \frac{\kappa }{2} + \frac{\delta}{2} \frac{M}{\sin M} \right) \sqrt{\frac{\kappa+\delta}{\delta} } \|\theta^{\rm in} \|_{\dot{H}_s(\Omega)}.\nonumber
\end{align}
Again, zero average phase, the bounds on the essential infimum and essential supremum of $\theta^\delta$ in \eqref{C-27} immediately imply
\begin{align} \label{C-32}
	\|\theta^{\delta}(t)\|_{L^\infty(\Omega)}^2\le C.
	\end{align}
Since the domain $\Omega$ is bounded, this also yields a uniform bound for $\theta^\delta(t)$ in the $L^2$-norm.	 With the above estimates \eqref{C-30}, \eqref{C-31}, and \eqref{C-32}, we again are ready to use Aubin--Lions lemma \ref{L2.1} in the next subsection.

\subsection{Passage to the limit for $\{ \theta^{\delta} \}$.}\label{sec:3.5}
In this subsection, we show that the sequence  $\{ \theta^{\delta} \}$ has a convergent sequence and its limit is in fact the global weak solution to \eqref{A-2}. Note that we have the embeddings
\[
H^s(\Omega)\hookrightarrow \hookrightarrow  L^2(\Omega)\hookrightarrow (H^{s}(\Omega))^\ast.
\]Applying Aubin--Lions lemma \ref{L2.1} once more with \eqref{C-30}, \eqref{C-31}, \eqref{C-32}, there exists a subsequence $\delta_j\to0$ and
\[
\theta\in L^2([0,T);H^s(\Omega))\cap C([0,T);L^2(\Omega))
\]
such that
\begin{equation*}
	\theta^{\delta_j}\to \theta \quad\text{strongly in } C([0,T);L^2(\Omega)),
	\qquad
	\theta^{\delta_j}\rightharpoonup \theta \quad\text{weakly in } L^2([0,T);H^s(\Omega)).
\end{equation*}
Moreover, we have
\begin{equation*}
	\operatorname*{ess\,sup}_{x\in \Omega} \theta(t,\cdot)\le \operatorname*{ess\,sup}_{x\in \Omega} \theta^{\mathrm{in}},\qquad
	\operatorname*{ess\,inf}_{x\in \Omega} \theta(t,\cdot)\ge \operatorname*{ess\,inf}_{x\in \Omega} \theta^{\mathrm{in}}.
\end{equation*}
By Proposition \ref{PA.3} we have 
\begin{equation*}
	{\mathcal D}[\theta(t)] \le {\mathcal D}[\theta^{\mathrm{in}}] <\pi,
	\qquad t\ge0.
\end{equation*}
Next, we show that the subsequence $\theta^{\delta_j}$ converges to the solution to \eqref{A-2}.
\begin{lemma}\label{L3.5}
	Let $T>0$ be fixed. For each $\delta>0$, let $\theta^\delta$ be a weak solution to \eqref{C-18}, and suppose that the following uniform in $\delta$ bound holds: 
	\begin{equation}\label{C-33}
		\sup_{\delta\in(0,1)}\Big(
		\|\theta^\delta\|_{L^\infty([0,T); L^2(\Omega))}
		+\|\theta^\delta\|_{L^2([0,T);H^s(\Omega))}
		+\|\partial_t\theta^\delta\|_{L^2([0,T);(H^{s}(\Omega))^\ast)}
		\Big)\le C.
	\end{equation}
	Then, there exists a weak solution $\theta$ to the following Cauchy problem:
	\begin{equation}\label{C-34}
	\begin{cases}
	\displaystyle \partial_t\theta(t,x)
		=\kappa\int_\Omega \frac{\sin\big(\theta(t,y)-\theta(t,x)\big)}{|x-y|^{d+2s}}\,\di y, \quad t > 0,\quad x\in\Omega, \\
	\displaystyle \theta \Big|_{t = 0} =\theta^{\mathrm{in}}.
	\end{cases}
	\end{equation}
	\end{lemma}
\begin{proof}
	By Aubin--Lions lemma \ref{L2.1} and assumption \eqref{C-33}, there exists a subsequence $\delta_j\to0$ and
	\[
	\theta\in L^2([0,T);H^s(\Omega))\cap C([0,T);L^2(\Omega))
	\]
	such that
	\begin{equation}\label{C-35}
		\theta^{\delta_j}\to\theta\quad \text{strongly in }C([0,T);L^2(\Omega)),\qquad
		\theta^{\delta_j}\rightharpoonup\theta\quad \text{weakly in }L^2([0,T);H^s(\Omega)).
	\end{equation}
	Since the proof is similar to Lemma \ref{L3.3}, we postpone it to Appendix \ref{App-B}. 
\end{proof}
Finally, we combine Lemma \ref{L3.3} and Lemma \ref{L3.5} to complete the proof of Theorem \ref{T2.1}.

\vspace{0.5cm}

\section{Proof of Theorem \ref{T2.2}}\label{sec:4}
\setcounter{equation}{0}
In this section, we study the long-time behavior of the weak solution to \eqref{A-2}. We split the proof of Theorem \ref{T2.2} into three steps: \vspace{0.1cm}
\begin{itemize}
\item
First, we use \eqref{B-1} to get the energy identity \eqref{C-38}. 
\vspace{0.1cm}
\item
Second, we use the fact that  $\sin x/ x$ is nonincreasing on $[0,\pi)$ to obtain \eqref{C-39}.
\vspace{0.1cm}
\item
Finally, we combine fractional Poincar\'e inequality and Gr\"onwall's inequality to find the desired estimates. 
\end{itemize}
In what follows, we perform the above procedures one by one. \newline

\noindent$\diamond$	\textbf{Step A} (Derivation of energy identity in $L^2$):	We set
	\begin{equation} \label{C-36}
	\tilde{\theta}(t,x):=\theta(t,x)-\bar{\theta}.
	\end{equation}
	Then, it is easy to see that 
\[ \int_\Omega \tilde{\theta}(t,x) \di x=0\ \ \text{for all }t\ge0. \]
Then, one has 
	\begin{align}\label{C-38}
	\frac12\frac{\di}{\di t}\| \tilde{\theta}(t)\|_{L^2(\Omega)}^2
	=\frac{\kappa}{2}\iint_{\Omega\times \Omega}
	\frac{\sin(\theta(t,y)-\theta(t,x))\big(\theta(t,x)-\theta(t,y)\big)}{|x-y|^{d+2s}} \di y\di x.
	\end{align}
	
	\vspace{0.2cm}
	
	\noindent$\diamond$	\textbf{Step B} (Rough monotonicity estimate):~By \eqref{C-36} and Theorem \ref{T2.1}, for all $t\ge0$ and a.e.\ $x,y\in\Omega$, it holds
	\[
	|\theta(t,x)-\theta(t,y)|\le M<\pi.
	\]
	Recall that 
	\[
	c_M:=\inf_{|z|\le M}\frac{\sin z}{z}=\frac{\sin M}{M}>0.
	\]
	Then for all $|z|\le M$ we have 
	\[ z\sin z\ge c_M z^2. \]
	 Applying this together with
	$z=\theta(t,x)-\theta(t,y)$ in \eqref{C-38}, one has 
	\begin{align}\label{C-39}
		\begin{aligned}
		\frac12\frac{\di}{\di t}\| \tilde{\theta}(t)\|_{L^2(\Omega)}^2 
		& \le -\frac{\kappa c_M}{2}\iint_{\Omega\times \Omega}
		\frac{(\theta(t,x)-\theta(t,y))^2}{|x-y|^{d+2s}} \di x\di y \\
		&= -\frac{\kappa c_M}{2}\iint_{\Omega\times \Omega}
		\frac{(\tilde{\theta}(t,x)-\tilde{\theta} (t,y))^2}{|x-y|^{d+2s}} \di x\di y.
		\end{aligned}
	\end{align}
	\vspace{0.2cm}
	\noindent$\diamond$	\textbf{Step C} (Deriving exponential relaxation): We use fractional Poincar\'e inequality and Gr\"onwall's inequality to derive the desired exponential decay estimate from \eqref{C-39}. Since $\int_\Omega \tilde{\theta}(t)=0$ for all $t$, the fractional Poincar\'e inequality implies the existence of
	$C_P=C_P(\Omega,d,s)>0$ such that for a.e.\ $t$,
	\begin{equation}\label{C-40}
		\iint_{\Omega\times \Omega} \frac{(\tilde{\theta}(t,x)-\tilde{\theta}(t,y))^2}{|x-y|^{d+2s}} \di x\di y
		\ge C_P\,\| \tilde{\theta}(t)\|_{L^2(\Omega)}^2.
	\end{equation}
	Finally, we combine \eqref{C-39} and \eqref{C-40} to get 
	\[
	\frac{\di}{\di t}\| \tilde{\theta}(t)\|_{L^2(\Omega)}^2
	\le -\kappa c_M C_P \| \tilde{\theta}(t)\|_{L^2(\Omega)}^2.
	\]
	By Gr\"onwall's lemma, we obtain the desired estimate:
	\[
	\| \tilde{\theta}(t)\|_{L^2(\Omega)}^2 \le e^{-\kappa c_M C_P t}\| \tilde{\theta}(0)\|_{L^2(\Omega)}^2
	= e^{-\kappa c_M C_P t}\|\theta^{\rm in}-\bar{\theta}\|_{L^2(\Omega)}^2.
	\]
	This completes the proof of Theorem \ref{T2.2}.

\vspace{0.5cm}

\section{Conclusion}\label{sec:5}
\setcounter{equation}{0}
In this paper, we have investigated the exponential relaxation of the continuum Kuramoto model with a non-integrable nonlocal spatial kernel for a constant natural frequency function, which introduces substantial analytical difficulties. In particular, the lack of integrability destroys relaxation estimates based on the classical diameter control techniques that play a crucial role in the analysis of synchronization phenomena as it is.
To overcome these apparent difficulties, we have introduced a double regularization procedure, which allowed us to construct global weak solutions to the continuum model. Furthermore, we established a diameter estimate via a novel truncation technique and verified that the phase diameter of solutions is nonincreasing provided that the initial phase diameter is smaller than $\pi$. Finally, for a constant natural frequency function, we rigorously showed that the solution converges exponentially to the initial mean phase in $L^2$-norm. Our result provided a rigorous mathematical justification for synchronization in the presence of nonlocal non-integrable spatial interactions. Of course, there are several interesting problems to be investigated further. To name a few, first, it would also be interesting to study the pointwise convergence of solutions and to characterize asymptotic synchronization profile. Second, the uniqueness of weak solution under non-integrable nonlocal spatial kernel still needs to be further explored. Third, the well-posedness and emergent behaviors of the continuum Kuramoto model with a non-integrable kernel under heterogeneous frequencies are challenging and relevant directions for a future research. We leave these interesting questions for a follow-up work.

\section*{Conflict of interest statement}
The authors declare no conflicts of interest.

\vspace{0.5cm}

\section*{Data availability statement}
The data supporting the findings of this study is available from the corresponding author upon reasonable request.

\section*{Ethical statement}
The authors declare that this manuscript is original, has not been published before, and is not currently being considered for publication elsewhere. The study was conducted by the principles of academic integrity and ethical research practices. All sources and contributions from others have been properly acknowledged and cited. The authors confirm that there is no fabrication, falsification, plagiarism, or inappropriate manipulation of data in the manuscript.
\appendix

\vspace{0.5cm}

\section{A global existence of doubly regularized problem}\label{App-A}
\setcounter{equation}{0}
In this appendix, we study a global solvability to the Cauchy problem for the doubly regularized equation \eqref{B-3-0}. \newline

Fix $T>0$, $\varepsilon>0$, and $\delta\ge0$, we set 
	\[
	\psi_\varepsilon(x,y) :=\frac{1}{(|x-y|+\varepsilon)^{d+2s}},\qquad
	(G_\varepsilon(u))(x):=\int_\Omega \psi_\varepsilon(x,y)\sin\big(u(y)-u(x)\big) \di y.
	\]
	Let $A$ be the nonnegative self-adjoint operator defined on its domain $D(A) \subset H^s (\Omega) \subset L^2(\Omega)$ associated with the
	symmetric bilinear form:
	\[
	\mathcal {A}(u,v):=\frac12\iint_{\Omega\times\Omega}\frac{(u(x)-u(y))(v(x)-v(y))}{|x-y|^{d+2s}}\di x \di y,
	\]
	so that $Au$ is (the realization of) the regional fractional Laplacian $(-\Delta)^s_\Omega u$. In this setting, the Cauchy problem for the doubly regularized equation \eqref{B-3-0} can be rewritten as follows. 
	\begin{equation} \label{Ap-1}
	\begin{cases}
	\displaystyle \partial_t\theta^{\varepsilon,\delta}(t)+\delta A\theta^{\varepsilon,\delta}(t)
		=\kappa\,G_\varepsilon(\theta^{\varepsilon,\delta}(t)), \quad t > 0,~~x \in \Omega, \\[1em]
	\displaystyle \theta^{\varepsilon,\delta} \Big|_{t = 0} =\theta^{\mathrm{in}}.
	\end{cases}
	\end{equation}

 In the following two propositions, we study a global well-posedness of \eqref{Ap-1} and phase diameter of $(\theta^{\varepsilon,\delta})$. 
\begin{proposition}[Well-posedness of the $(\varepsilon,\delta)$-regularized problem]\label{PA.1}
The following assertions hold.
\begin{enumerate}
\item
For $\theta^{\mathrm{in}}\in L^2(\Omega)\cap H^{s}(\Omega)$, the Cauchy problem \eqref{Ap-1} admits a unique mild solution
	\[
	\theta^{\varepsilon,\delta}\in C([0,T);L^2(\Omega))\cap L^\infty( [0,T);H^s(\Omega)).
	\]
\item
For $\theta^{\mathrm{in}}\in L^\infty(\Omega)\cap H^{s}(\Omega)$, the mild solution obtained in (1) satisfies the following regularity:
\[ \theta^{\varepsilon,\delta}\in C([0,T);L^\infty(\Omega))\cap L^\infty([0,T);H^s(\Omega)). \]
\end{enumerate}
\end{proposition}
\begin{proof}
We split the proof into four steps:
\vspace{0.1cm}
\begin{itemize}
\item
First, we show that $G_\varepsilon$ is globally Lipschitz.
\vspace{0.1cm}
\item
Second, we recall the semigroup generated by regional fractional Laplacian. 
\vspace{0.1cm}
\item
Third, we show the existence and uniqueness of the solution in $L^2(\Omega)$ by fixed-point theorem.
\vspace{0.1cm}
\item
Finally, we show the existence and uniqueness of the solution in $L^{\infty}(\Omega)$. 
\end{itemize}
Now, we proceed with the proof as follows.\vspace{0.2cm}
	
\noindent$\diamond$ \textbf{Step A} ($G_\varepsilon$ is globally Lipschitz on $L^2(\Omega)$ and $L^\infty(\Omega)$):~We first claim that, for each fixed $\varepsilon>0$,
	\begin{equation}\label{A.1}
		K_\varepsilon := \sup_{x\in\Omega}\int_\Omega \psi_\varepsilon^2(x,y) \di y<\infty.
	\end{equation}
	Indeed, since $\Omega$ is bounded, for any $x\in\Omega$ we have
	\[
	\int_\Omega \frac{1}{(|x-y|+\varepsilon)^{2d+4s}} \di y
	\le \int_{\Omega}\frac{1}{\varepsilon^{2d+4s}} \di y
	= |\Omega| \varepsilon^{-2d-4s}.
	\]
	This proves \eqref{A.1}. Similarly, we obtain
	\begin{align} \label{A.2}
	K_\varepsilon^\ast := \sup_{x\in \Omega} \int_\Omega \psi_\varepsilon(x,y) \di y <\infty.
	\end{align}
	Let $u$ and $v$ be measurable functions. Then, we use
	\[ |\sin a-\sin b|\le |a-b| \]
	to see that  for a.e.\ $x\in\Omega$,
	\begin{align}
	\begin{aligned} \label{A.2-1}
		&|G_\varepsilon(u)(x)-G_\varepsilon(v)(x)| \\
		& \hspace{1cm} \le \int_\Omega \psi_\varepsilon(x,y)
		\Big|\sin\big(u(y)-u(x)\big)-\sin\big(v(y)-v(x)\big)\Big| \di y\\
		& \hspace{1cm} \le \int_\Omega \psi_\varepsilon(x,y)\Big(|u(y)-v(y)|+|u(x)-v(x)|\Big) \di y\\
		& \hspace{1cm} \le \int_\Omega \psi_\varepsilon(x,y)|u(y)-v(y)| \di y
		+\Big(\int_\Omega \psi_\varepsilon(x,y) \di y\Big)\,|u(x)-v(x)| \\
		& \hspace{1cm} =: {\mathcal I}_{11}(x) +  {\mathcal I}_{12}(x).
	\end{aligned}
	\end{align}
Then, it is easy to see that 
\begin{equation} \label{A.2-1-1}
|G_\varepsilon(u)(x)-G_\varepsilon(v)(x)|^2 \leq 2 \Big( |{\mathcal I}_{11}(x)|^2 +  |{\mathcal I}_{12}(x)|^2 \Big). 
\end{equation}
\noindent Below, we estimate the terms ${\mathcal I}_{1i},~i=1,2$ one by one. 	\newline

\noindent $\clubsuit$~(Estimate of ${\mathcal I}_{11})$: We use the Cauchy-Schwarz inequality to see that for $x \in \Omega$, 
\[
	|{\mathcal I}_{11}(x)| = \int_\Omega \psi_\varepsilon(x,y)|u(y)-v(y)|\di y \le \Big(\int_\Omega \psi^2_\varepsilon(x,y)\di y\Big)^{1/2}\|u-v\|_{L^2(\Omega)} \le (K_\varepsilon)^{1/2}\|u-v\|_{L^2(\Omega)}.
\]
	This and boundedness of $\Omega$ yield
\begin{equation} \label{A.2-2}
\int_\Omega  |{\mathcal I}_{11}(x)|^2 \di x \leq  |\Omega| K_\varepsilon \|u-v\|^2_{L^2(\Omega)}. 
\end{equation}
\noindent $\clubsuit$~(Estimate of ${\mathcal I}_{12})$:~Again, we use \eqref{A.2} to see 
\[
|{\mathcal I}_{12}(x) |^2\leq |K_\varepsilon^\ast|^2 \,|u(x)-v(x)|^2.
\]
This yields 
\begin{equation} \label{A.2-3}
\int_\Omega  |{\mathcal I}_{12}(x)|^2 \di x \leq   |K_\varepsilon^\ast|^2 \|u-v\|^2_{L^2(\Omega)}. 
\end{equation}
We combine \eqref{A.2-1-1}, \eqref{A.2-2} and \eqref{A.2-3} to obtain 
	\begin{align*}
	\begin{aligned}
		\|G_\varepsilon(u)-G_\varepsilon(v)\|_{L^2(\Omega)}^2
		&\le 2|\Omega|K_\varepsilon\|u-v\|^2_{L^2(\Omega)}
		+2(K_\varepsilon^\ast)^2 \|u-v\|^2_{L^2(\Omega)} \notag\\
		&\le C(\Omega)\,(K_\varepsilon+ (K_\varepsilon^\ast)^2)\,\|u-v\|^2_{L^2(\Omega)},
	\end{aligned}
	\end{align*}
i.e.,
\begin{equation} \label{A.3}
\|G_\varepsilon(u)-G_\varepsilon(v)\|_{L^2(\Omega)} \leq  \Big( C(\Omega)\,(K_\varepsilon+ (K_\varepsilon^\ast)^2) \Big)^{\frac{1}{2}} \|u-v\|_{L^2(\Omega)}.
\end{equation}
Thus, the map $G_\varepsilon:L^2(\Omega)\to L^2(\Omega)$ is globally Lipschitz.
	\\[12pt]
	Similarly, we take the $L^\infty$-norm in $x$ and use \eqref{A.2} to find 
	\begin{equation}\label{A.4}
		\|G_\varepsilon(u)-G_\varepsilon(v)\|_{L^\infty(\Omega)}
		\le 2K_\varepsilon^\ast\,\|u-v\|_{L^\infty(\Omega)},
	\end{equation}
	so $G_\varepsilon$ is globally Lipschitz on $L^\infty(\Omega)$ as well.
	\vspace{0.2cm}

\noindent$\diamond$ \textbf{Step B} $(-\delta A$ generates an analytic semigroup on $L^2(\Omega)$):~
	By \eqref{NewL11} and \eqref{NewL12}, the bilinear form $\mathcal{A}$ corresponding to the regional fractional Laplacian generates a contraction Markovian semigroup $\left( e^{-tA} \right)_{t \ge 0}$. Therefore $-\delta A$ also generates an contraction Markovian semigroup $\left( e^{-\delta tA}\right)_{t\ge0}$.
	\vspace{0.2cm}

\noindent$\diamond$ \textbf{Step C} (Existence and uniqueness by a fixed-point argument in $L^2(\Omega)$):~We define the mild solution map on $C([0,T);L^2(\Omega))$ by
	\begin{equation*}
		(\mathcal T\theta)(t):=e^{-\delta tA}\theta^{\mathrm{in}}
		+\kappa\int_0^t e^{-\delta (t-\tau)A}\,G_\varepsilon(\theta(\tau))\,\di \tau.
	\end{equation*}
	Using the contraction property $\|e^{-\delta tA}\|_{\mathcal L(L^2,L^2)}\le 1$ and
	the Lipschitz bound \eqref{A.3}, we obtain for $\theta,\varphi\in C([0,T);L^2(\Omega))$
	\begin{align*}
	\begin{aligned}
	& \|(\mathcal T\theta)(t)-(\mathcal T\varphi)(t)\|_{L^2(\Omega)}^2 \\
	& \hspace{0.5cm} \le \kappa^2 t\int_0^t \|G_\varepsilon(\theta(\tau))-G_\varepsilon(\varphi(\tau))\|_{L^2(\Omega)}^2\di \tau \le  \kappa^2 T L_\varepsilon^2 \int_0^t \|\theta(\tau)-\varphi(\tau)\|_{L^2(\Omega)}^2\,\di \tau,
	\end{aligned}
	\end{align*}
	where $L_\varepsilon^2:=C(\Omega)(K_\varepsilon + (K_\varepsilon^\ast)^2)$. \newline
	
	We take the supremum over $t\in[0,T]$ to find 
	\[
	\|\mathcal T\theta-\mathcal T\varphi\|_{C([0,T);L^2(\Omega))}^2
	\le (\kappa L_\varepsilon T)^2 \,\|\theta-\varphi\|_{C([0,T);L^2(\Omega))}^2.
	\]
	We choose $T>0$ such that 
	\[ \kappa L_\varepsilon T <1. \]
	Then, we obtain that $\mathcal T$ is a contraction on
	$C([0,T);L^2(\Omega))$, hence there exists a unique fixed point
	$\theta^{\varepsilon,\delta}\in C([0,T);L^2(\Omega))$, which is the unique mild solution \cite{Rankin1993}.
	By iterating the argument on consecutive subintervals (using global Lipschitz continuity of
	$G_\varepsilon$), the mild solution extends uniquely to any finite-time interval since $\mathcal{T}(\theta)$ is bounded on such intervals. That is 
	\begin{align*}
		\|(\mathcal T\theta)(t)\|_{L^2(\Omega)}^2&\le 2\|e^{-\delta tA}\theta^{\mathrm{in}}\|_{L^2(\Omega)}^2
		+2\kappa^2 t \int_0^t \|e^{-\delta (t-\tau)A}\,G_\varepsilon(\theta(\tau))\|_{L^2(\Omega)}^2\,\di \tau\\&\le 2\|\theta^{\mathrm{in}}\|_{L^2(\Omega)}^2+2 T^2 C(d,s,\Omega)\,\varepsilon^{-2s}.
	\end{align*}

	\noindent$\diamond$ \textbf{Step D} ($L^\infty$ continuity):~If $\theta^{\mathrm{in}}\in L^\infty(\Omega)$, we work in the Banach space $L^\infty(\Omega)$.
	The bound \eqref{A.4} shows that $G_\varepsilon$ is globally Lipschitz on $L^\infty$,
	and the semigroup $\{e^{-\delta tA}\}$ is bounded on $L^\infty$. That is 
	\[ \|e^{-\delta tA}\|_{\mathcal L(L^{\infty},L^{\infty})}\le 1 \]
	and 
	\begin{align*}
		\|(\mathcal T\theta)(t)\|_{L^{\infty} (\Omega)}&\le \|e^{-\delta tA}\theta^{\mathrm{in}}\|_{L^\infty (\Omega)}
		+\kappa\int_0^t \|e^{-\delta (t-\tau)A}\,G_\varepsilon(\theta(\tau))\|_{L^\infty (\Omega)}\,\di \tau\\
		&\le \|\theta^{\mathrm{in}}\|_{L^\infty (\Omega)}+T C(d,s,\Omega)\,\varepsilon^{-2s}.
	\end{align*} Hence the same fixed point argument yields a unique mild
	solution in $C([0,T);L^\infty(\Omega))$. Uniqueness implies that this $L^\infty$-mild solution
	coincides with the $L^2$-mild solution constructed above whenever
	$\theta^{\mathrm{in}}\in L^2\cap L^\infty$.
\end{proof}

\begin{remark}
Since mild solutions are also weak solutions, we obtain uniqueness and existence of weak solutions in $C([0,T); L^2(\Omega))$ or in $C([0,T); L^\infty(\Omega))$, depending on the regularity of initial data. Using the estimates \eqref{C-15-1} and \eqref{C-16}, we also see that the solution is in $L^\infty( [0,T); H^s(\Omega))$.
\end{remark} 
\begin{proposition} \label{PA.2}
Fix $T>0$ and $\delta > 0$, let $\Omega\subset\mathbb R^d$ be a measurable set with $0<|\Omega|<\infty$. We assume that $\{\theta^{\varepsilon,\delta}\}_{\varepsilon>0}\subset C([0,T);L^2(\Omega))$ and $\theta^\delta\in C([0,T);L^2(\Omega))$ satisfy the following conditions:
\begin{enumerate}
\item
There exists a sequence $(\varepsilon_j)$ tending to $0$ as $j \to \infty$ such that 
\[ \theta^{\varepsilon_j,\delta}\to\theta^\delta \quad \text{strongly in } C([0,T);L^2(\Omega)), \quad j \to \infty. \]
\item		
There exist constants $M_\delta,m_\delta\in\mathbb R$ such that for all $j$ and all $t\in[0,T)$,
		\[
		\operatorname*{ess\,sup}_{x\in \Omega} \theta^{\varepsilon_j,\delta}(t,x)\le M_\delta,
		\quad
		\operatorname*{ess\,inf}_{x\in \Omega} \theta^{\varepsilon_j,\delta}(t,x)\ge m_\delta.
		\]
\end{enumerate}
Then, for all $t\in[0,T]$, one has 
		\[
		\operatorname*{ess\,sup}_{x\in \Omega} \theta^\delta(t,x)\le M_\delta,
		\quad
		\operatorname*{ess\,inf}_{x\in \Omega} \theta^\delta(t,x)\ge m_\delta.
		\]
\end{proposition}
\begin{proof} Since the proof will be the same as the proof of Proposition \ref{PA.3}, we omit it.
\end{proof}

\begin{proposition} \label{PA.3}
	Fix $T>0$ and let $\Omega\subset\mathbb R^d$ be a measurable set with $0<|\Omega|<\infty$. We assume that $\{\theta^{\delta}\} \subset C([0,T);L^2(\Omega))$ and 
		$\theta\in C([0,T);L^2(\Omega))$ satisfy the following relations:
\begin{enumerate}
\item		
There exists a sequence $(\delta_j)$ tending to $0$ as $j \to \infty$ such that 
		\[
		\theta^{\delta_j}\to \theta \quad \text{strongly in } C([0,T);L^2(\Omega))
		\quad \quad j \to \infty.
		\]
\item		
		There exist constants $M,m\in\mathbb R$ such that for all $j$ and all $t\in[0,T)$,
		\[
		\operatorname*{ess\,sup}_{x\in \Omega}\theta^{\delta_j}(t,x)\le M,
		\quad
		\operatorname*{ess\,inf}_{x\in \Omega}\theta^{\delta_j}(t,x)\ge m.
		\]
\end{enumerate}		
		Then, the following assertions hold.
\begin{enumerate}
\item		
For all $t\in[0,T)$, one has 
\[ \operatorname*{ess\,sup}_{x\in \Omega}\theta(t,x)\le M, \quad \operatorname*{ess\,inf}_{x\in \Omega}\theta(t,x)\ge m, \quad {\mathcal D}[\theta(t)] \le M - m.\]
\item
If $M-m<\pi$, then
	\[
	{\mathcal D}[\theta(t)] <\pi \quad \text{for all } t\in[0,T).
\]
\end{enumerate}
\end{proposition}
\begin{proof}
The second assertion follows from the last estimate in the first assertion. Hence, we focus on the first assertion. We fix $t\in[0,T]$ and split the proof into two steps. \newline

\noindent$\bullet$ \textbf{Step A} (Upper bound):~Suppose the contrary holds, i.e.,
\[  \operatorname*{ess\,sup}_{x\in \Omega}\theta(t,x)>M. \]
	Then there exists $\eta>0$ such that the superlevel set
	\[
	A_\eta:=\{x\in\Omega:\ \theta(t,x)>M+\eta\}
	\]
	has positive measure:
	\[ |A_\eta|>0. \]
	By the uniform bound on $\theta^{\delta_j}$, we have 
	\[ \theta^{\delta_j}(t,x)\le M \quad \mbox{for a.e.\ $x\in\Omega$}, \]
	hence for a.e.\ $x\in A_\eta$,
	\[
	\theta(t,x)-\theta^{\delta_j}(t,x)\ge (M+\eta)-M=\eta.
	\]
	Therefore, for all $j$, it holds
	\[
	\|\theta(t)-\theta^{\delta_j}(t)\|_{L^2(\Omega)}^2
	\ge \int_{A_\eta}(\theta(t,x)-\theta^{\delta_j}(t,x))^2 \di x
	\ge \eta^2|A_\eta|>0,
	\]
	which is contradictory to
	\[ \|\theta(t)-\theta^{\delta_j}(t)\|_{L^2(\Omega)} \to 0, \]
	as $j\to\infty$, a consequence of the strong convergence in $C([0,T);L^2(\Omega))$. \newline
	
	\noindent$\bullet$ \textbf{Step B} (Lower bound):~Suppose the contrary holds, i.e., 
	\[ \operatorname*{ess\,inf}_{x\in \Omega}\theta(t,x)<m. \]
	Then there exists $\eta>0$ such that the sublevel set
	\[
	B_\eta:=\{x\in\Omega:\ \theta(t,x)<m-\eta\}
	\]
	has a positive measure. Since $\theta^{\delta_j}(t,x)\ge m$ a.e., we have 
	\[
	\theta^{\delta_j}(t,x)-\theta(t,x)\ge m-(m-\eta)=\eta, \quad \mbox{for a.e.\ $x\in B_\eta$}, 
	\]
	and thus
	\[
	\|\theta^{\delta_j}(t)-\theta(t)\|_{L^2(\Omega)}^2\ge \eta^2|B_\eta|>0,
	\]
	which is again contradictory to the strong $L^2$-convergence at time $t$. For the last estimate, we combine the results in {\bf Step A} and {\bf Step B} to see
	\[
	{\mathcal D}[\theta(t)] = \operatorname*{ess\,sup}_{x\in \Omega}\theta(t,x) -  \operatorname*{ess\,inf}_{x\in \Omega}\theta(t,x) \leq M - m.
	\]
	Since $t\in[0,T)$ was arbitrary, the conclusion holds for all $t$.
\end{proof}

\vspace{0.5cm}

\section{Proof of Lemma \ref{L3.5}}\label{App-B}
\setcounter{equation}{0}
Since the proof is very lengthy, we split the proof into five steps.\newline 

\noindent$\bullet$ \textbf{Step A} (Weak formulation and symmetrization):~Let $\varphi\in C_c^\infty([0,T)\times\Omega)$. We multiply \eqref{C-18} by $\varphi$ and use integration by parts in time to see that for every $\delta>0$,
\begin{align}
\begin{aligned} \label{C.1}
	&-\int_0^T\!\!\int_\Omega \theta^\delta\,\partial_t\varphi\,\di x\di t
	-\int_\Omega \theta^{\rm in}(x)\varphi(0,x)\,\di x
	+\delta\int_0^T \mathcal A(\theta^\delta(t),\varphi(t))\,\di t \\
	&\hspace{2cm} = \kappa\int_0^T\!\!\iint_{\Omega\times \Omega}
	\frac{\sin(\theta^\delta(t,y)-\theta^\delta(t,x))}{|x-y|^{d+2s}}\,
	\varphi(t,x)\,\di y\di x\di t, 
\end{aligned}
\end{align}
where $\mathcal{A}$ is defined as in \eqref{B-3}. Using the oddness of the sine function and exchanging $x$ and $y$, we may symmetrize the right-hand side as
\begin{align}
\begin{aligned} \label{C.2}
& \iint_{\Omega\times \Omega} \frac{\sin(\theta^\delta(y)-\theta^\delta(x))}{|x-y|^{d+2s}}\varphi(x)\,\di x\di y \\
& \hspace{2cm} =\frac12\iint_{\Omega\times \Omega}
	\frac{\sin(\theta^\delta(y)-\theta^\delta(x))}{|x-y|^{d+2s}}\big(\varphi(x)-\varphi(y)\big)\,\di x\di y.
\end{aligned}
\end{align}
Hence, we can combine \eqref{C.1} and \eqref{C.2} to find 
\begin{align}
	&-\int_0^T\!\!\int_\Omega \theta^\delta\,\partial_t\varphi\,\di x\di t
	-\int_\Omega \theta^{\rm in}(x)\varphi(0,x)\,\di x
	+\delta\int_0^T \mathcal A(\theta^\delta(t),\varphi(t))\,\di t \notag\\
	&\hspace{2cm} = \frac{\kappa}{2}\int_0^T\!\!\iint_{\Omega\times \Omega}
	\frac{\sin(\theta^\delta(t,y)-\theta^\delta(t,x))}{|x-y|^{d+2s}}\,
	\big(\varphi(t,x)-\varphi(t,y)\big)\,\di y\di x\di t. \label{C.3}
\end{align}
\noindent$\bullet$ \textbf{Step B} (Limit in time and initial terms):~By the strong convergence in \eqref{C-35}, we obtain
\[
\int_0^T\!\!\int_\Omega \theta^{\delta_j}\,\partial_t\varphi\,\di x\di t
\to
\int_0^T\!\!\int_\Omega \theta\,\partial_t\varphi\,\di x\di t,
\qquad
\int_\Omega \theta^{\delta_j}(0)\varphi(0)\,\di x
=
\int_\Omega \theta^{\rm in}\varphi(0)\,\di x.
\]
\noindent$\bullet$ \textbf{Step C} (The vanishing of the $\delta$--diffusion term):~Recall
\[
\mathcal A(u,\varphi)
:=\frac12\iint_{\Omega\times\Omega}
\frac{(u(x)-u(y))(\varphi(x)-\varphi(y))}{|x-y|^{d+2s}}\,\di x\di y.
\]
By the Cauchy-Schwarz inequality, it holds
\[
|\mathcal A(u,\varphi)|\le \frac{1}{2}\|u\|_{H^s(\Omega)}\|\varphi\|_{H^s(\Omega)}.
\]
Therefore, we have
\begin{align*}
	\left|\delta_j\int_0^T \mathcal A(\theta^{\delta_j}(t),\varphi(t))\,\di t\right|
	&\le \frac{\delta_j}{2}\|\theta^{\delta_j}\|_{L^2([0,T);H^s(\Omega))}\,\|\varphi\|_{L^2([0,T);H^s(\Omega))}\\
	&\le C\delta_j\|\varphi\|_{L^2([0,T);H^s(\Omega))}\xrightarrow[j\to\infty]{}0,
\end{align*}
where we used the uniform bound \eqref{C-33}. Hence, the contribution by diffusion disappears in the limit. \newline

\noindent$\bullet$ \textbf{Step D} (Limit in the singular Kuramoto term):~For a.e.\ $t\in(0,T)$, we set
\begin{equation} \label{C.3-1}
\mathcal K(u,\varphi)(t)
:=\iint_{\Omega\times \Omega}
\frac{\sin(u(t,y)-u(t,x))}{|x-y|^{d+2s}}\big(\varphi(t,x)-\varphi(t,y)\big)\,\di x\di y.
\end{equation}
We claim that
\begin{equation}\label{C.4}
	\int_0^T \mathcal K(\theta^{\delta_j},\varphi)(t)\,\di t
	\to
	\int_0^T \mathcal K(\theta,\varphi)(t)\,\di t.
\end{equation}
Fix $\rho\in(0,1)$, we decompose $\Omega\times\Omega$ into $\{|x-y|>\rho\}\cup\{|x-y|\le\rho\}$, writing accordingly
\[ \mathcal K = \mathcal K^{>\rho}+\mathcal K^{\le\rho}. \]
Next, we consider two cases. \newline

\noindent$\diamond$ \textbf{Case D.1}:~On $\{|x-y|>\rho\}$, we have
\[
\Big| \frac{\sin(u(t,y)-u(t,x))}{|x-y|^{d+2s}} \Big| \leq \rho^{-d-2s}.
\]
Thus, the integrand in \eqref{C.3-1} is dominated by an $L^1$-function.
Moreover, it follows from \eqref{C-35} that there exists a subsequence $\delta_j \to 0$ such that 
\[ \theta^{\delta_j}(t,x)\to\theta(t,x) \quad \mbox{a.e.\ in $(0,T)\times\Omega$}. \]
Hence, we have
\[
\sin(\theta^{\delta_j}(t,y)-\theta^{\delta_j}(t,x))\to
\sin(\theta(t,y)-\theta(t,x))\quad\text{a.e.\ in }(0,T)\times\Omega\times\Omega.
\]
Thus, by dominated convergence theorem (for fixed $\rho$), it holds
\begin{equation}\label{C.5}
	\int_0^T \mathcal K^{>\rho}(\theta^{\delta_j},\varphi)(t)\,\di t
	\to
	\int_0^T \mathcal K^{>\rho}(\theta,\varphi)(t)\,\di t.
\end{equation}

\smallskip
\noindent$\diamond$ \textbf{Case D.2.} ~On $\{|x-y| \leq \rho\}$, we use
\[ |\sin a|\le |a| \quad \mbox{and} \quad |\varphi(t,x)-\varphi(t,y)|\le \|\nabla\varphi\|_{L^\infty((0,T)\times\Omega)}|x-y| \]
to estimate
\begin{align*}
\begin{aligned}
	&\big|\mathcal K^{\le\rho}(\theta^{\delta},\varphi)(t)\big| \le \|\nabla\varphi\|_{L^\infty((0,T)\times\Omega)}\int_{\{|x-y|\le\rho\}} \frac{|\theta^\delta(t,y)-\theta^\delta(t,x)|}{|x-y|^{d+2s-1}}\,\di x\di y
	\\
	& \hspace{1cm} \le \|\nabla\varphi\|_{L^\infty((0,T)\times\Omega)} \| \theta^\delta (t)\|_{\dot H^s (\Omega)}
	\left(\iint_{\{|x-y|\le\rho\}}\frac{1}{|x-y|^{d+2s-2}}\,\di x\di y\right)^{1/2}.
\end{aligned}
\end{align*}
By \eqref{C-30}, the Gagliardo seminorm of $\theta^{\delta}(t)$ can be bounded by a constant independent of $\delta$. For the last term, using $z=x-y$ and polar coordinates, we obtain
\[
\iint_{\{|x-y|\le\rho\}}\frac{1}{|x-y|^{d+2s-2}}\,\di x\di y
\le C_\Omega\int_{|z|\le\rho}\frac{1}{|z|^{d+2s-2}}\,dz
= C_{\Omega,d}\int_0^\rho r^{1-2s}\,dr
= C_{\Omega,d} \,\rho^{2-2s},
\]
which tends to $0$ as $\rho\to0$ since $s\in(0,1)$.
Integrating in time and using the uniform bound \eqref{C-33}, we obtain the uniform smallness estimate
\begin{equation}\label{C.6}
	\sup_{j\in\mathbb N}\int_0^T \big|\mathcal K^{\le\rho}(\theta^{\delta_j},\varphi)(t)\big|\,\di t
	\le C\,\rho^{1-s}\qquad\text{for all }\rho\in(0,1).
\end{equation}
The same estimate holds with $\theta$ in place of $\theta^{\delta_j}$. Let $j\to\infty$ in \eqref{C.5}, and then let $\rho\to0$ using \eqref{C.6}.
This yields \eqref{C.4}.

\medskip
\noindent$\bullet$ \textbf{Step E} (Limiting weak formulation):~Passing to the limit $j\to\infty$ in \eqref{C.3} using \textbf{Step~A}--\textbf{Step D}, we conclude that for every
$\varphi\in C_c^\infty([0,T)\times\Omega)$,
\begin{align*}
	&-\int_0^T\!\!\int_\Omega \theta\,\partial_t\varphi\,\di x\di t
	-\int_\Omega \theta^{\rm in}(x)\varphi(0,x)\,\di x\notag\\
	&\hspace{1cm} = \frac{\kappa}{2}\int_0^T\!\!\iint_{\Omega\times \Omega}
	\frac{\sin(\theta(t,y)-\theta(t,x))}{|x-y|^{d+2s}}\,
	\big(\varphi(t,x)-\varphi(t,y)\big)\,\di y\di x\di t.
\end{align*}
This is exactly the weak formulation of \eqref{A-2}.

\vspace{0.5cm}

%
%
%
%
%
%
\bibliographystyle{amsplain}

\end{document}